\documentclass[a4paper,12pt]{amsart}
\usepackage[utf8]{inputenc}
\usepackage{amsmath}
\usepackage{amsfonts}
\usepackage{amsthm}
\usepackage{mathtools}
\usepackage{amssymb}
\usepackage{xcolor}
\usepackage[all]{xy}
\usepackage{enumerate}
\usepackage{tikz}
\usepackage{tikz}
\usepackage{bbm}
\usepackage{marvosym}
\usepackage{cleveref} 
\usepackage{tikz-cd} \usetikzlibrary{tikzmark,shapes.geometric,shapes.symbols}
\usepackage{cancel}
\usepackage[normalem]{ulem}
\usepackage{verbatim} 
\allowdisplaybreaks 

\newtheorem{theorem}{Theorem}[section]
\newtheorem{lemma}[theorem]{Lemma}
\newtheorem{proposition}[theorem]{Proposition}
\newtheorem{corollary}[theorem]{Corollary}

\theoremstyle{definition}
\newtheorem{definition}[theorem]{Definition}

\theoremstyle{remark}
\newtheorem{example}[theorem]{Example}
\newtheorem{remark}[theorem]{Remark}

\numberwithin{equation}{section}

\newcommand{\bool}[1]{[\![{#1}]\!]}

\newcommand{\Z}{\mathbb{Z}}

\newcommand{\kpar}{\Bbbk_{par}}

\title[The monoidal structure of the category of partial representations of finite groups]{The monoidal structure of the category of partial representations of finite groups}

\author[Arthur R. Alves Neto]{Arthur R. Alves Neto}

\author[E. Batista]{Eliezer Batista}

\author[J. Mendez]{Javier Méndez}

\address[Javier Méndez, Eliezer Batista and Arthur R. Alves Neto]{Departamento de Matem\'atica, Universidade Federal de Santa Catarina, 88040-970 Florian\'opolis SC, Brazil. }
\email{javier.mendeza92@gmail.com } \email{eliezer1968@gmail.com}  \email{arthurntcwb@yahoo.com.br}


\begin{document}

\begin{abstract}
    In this work, we analyze the structure of the category of partial representations of a finite group $G$ as a multifusion category, providing an alternative way to describe simple objects and their tensor products. We describe the interconnection between the category of partial representations of a finite group and the category of global representations of its subgroups (the Christmas Tree's Theorem). Also, for a finite abelian group $G$, we prove that the category of partial representations of any of its subgroups can be embedded into the category of partial representations of $G$ (the Matryoshka's Theorem).
\end{abstract}

\maketitle

\section{Introduction}

Partial group representations and the partial group algebra first appear in \cite{DEP}. The authors' main motivation was to find a mechanism for distinguishing between two abelian groups by means of algebras associated to them. It is known that, for a finite abelian group $G$ and for a field $\Bbbk$ whose characteristic does not divide the order of $G$, the group algebra $\Bbbk G$ is isomorphic as an algebra to a direct sum of copies of the basis field $\Bbbk$. Therefore, two abelian groups of the same order are indistinguishable from the point of view of their group algebras. 

Roughly speaking, a partial representation of a group can be obtained by restricting a representation of the same group to a subspace which is not invariant, but which satisfies some compatibility conditions \cite{Abadiedilations, ABVdilations}. Formally, while a representation $\rho :G\rightarrow GL(\mathbb{V})$ associates each element $g\in G$ to a linear automorphism $\rho (g): \mathbb{V} \rightarrow \mathbb{V}$, a partial representation $\pi :G \rightarrow \text{End}_\Bbbk (\mathbb{V})$, in its turn, associates each element $g\in G$ to a linear transformation $\pi (g) :\mathbb{V} \rightarrow \mathbb{V}$. This linear transformation can be interpreted as a partially defined linear isomorphism $\pi (g):\text{Dom}(\pi (g)) \rightarrow \text{Im}(\pi (g))$. The composition of two distinct partially defined isomorphisms in $\mathbb{V}$ is only possible under the assumption that their domains and codomains are compatible, this is done by projecting first on the domain of the first linear transformation or after on the codomain of the second. That is, for any elements $g,h \in G$ we have
\begin{eqnarray*}
\pi (g) \circ \pi (h) \circ \pi (h^{-1}) & = & \pi (gh) \circ \pi (h^{-1}), \\
\pi (g^{-1}) \circ \pi (g) \circ \pi (h) & = & \pi (g^{-1}) \circ \pi (gh) .
\end{eqnarray*}

All partial representations of a group $G$ can be interpreted as modules over an algebra constructed upon $G$, which is called the partial group algebra $\Bbbk_{par}G$ \cite{DEP}. This algebra is isomorphic to an algebra of a groupoid related to $G$, whose connected components are associated to subsets $X\subseteq G$ containing the neutral element of $G$. A simple counting argument shows us that for a group $G$ of order $n$, the dimension of $\Bbbk_{par}G$ is $d=2^{n-2}(n+1)$. The geometry of the groupoid in question helps us to better understand the structure of the algebra $\Bbbk_{par}G$ itself as a direct product of matrix algebras, each one with entries in the group algebras of a subgroup of $G$ which serves as isotropy subgroup of the vertices of one of the connected components of the groupoid \cite{DEP,DP}. The rich structure of the partial group algebras allows one to distinguish between two abelian groups of the same order, for example, while the abelian groups $\mathbb{Z}_4$ and $\mathbb{Z}_2 \times \mathbb{Z}_2$ both have their group algebras over the complex numbers isomorphic to $\mathbb{C}^4$, we have $\mathbb{C}_{par} \left( \mathbb{Z}_4 \right) \cong 7\mathbb{C} \oplus M_2 (\mathbb{C}) \oplus M_3 (\mathbb{C})$ and $\mathbb{C}_{par} \left( \mathbb{Z}_2 \times \mathbb{Z}_2 \right) \cong 11\mathbb{C} \oplus M_3 (\mathbb{C})$.

The aim of this article is to describe the properties of the category of partial modules over a finite group $G$. Our approach is to follow the general theory of partial representations of Hopf algebras \cite{ABVparreps}. The category of partial modules of a Hopf algebra $H$ is the theory of modules of a Hopf algebroid, denoted by $H_{par}$, constructed from $H$. For the case of a group algebra, $H=\Bbbk G$, the Hopf algebroid $(\Bbbk G)_{par}$ coincides with the partial group algebra $\Bbbk_{par} G$. More specifically, there is a subalgebra $A_{par}(H)\subseteq H_{par}$ which encodes the size of the allowable partiality for the Hopf algebra $H$. As an example, for the universal enveloping algebra of a Lie algebra $H= \mathcal{U} (\mathfrak{g})$, we have $A_{par}(H)=\Bbbk$, that is, this Hopf algebra does not admit any partial module which is not global. 

For the case of group algebras, $A_{par}(\Bbbk G)$ is a commutative algebra generated by idempotents. The algebra $\Bbbk_{par} G$ has a structure of Hopf algebroid over the base algebra $A_{par}(\Bbbk G)$. We use some techniques developed in \cite{ABRhpar} to obtain an alternative proof of the Theorem 3.2 of \cite{DEP} which states that the structure of $\Bbbk_{par}G$ is a product of matrix algebras with entries in the group algebras of subgroups $H\leq G$ (Proposition \ref{GammaX} and Theorem \ref{theorem.kparGGamma.iso.MnKG}). 

Another important aspect is that the partial group algebra $\Bbbk_{par}G$ is isomorphic to a groupoid algebra $\Bbbk \Gamma (G)$ (see Proposition \ref{gammaG} for a description of the groupoid structure of $\Gamma (G)$). Any groupoid algebra is endowed with a structure of a weak Hopf algebra and weak Hopf algebras can be understood as Hopf algebroids over a suitable subalgebra. A very natural question is whether the two Hopf algebroid structures coming from the theory of partial representations \cite{ABVparreps} and that coming from the weak Hopf algebra structure isomorphic? The affirmative answer is found in Theorem \ref{kpar-e-kgamma-SAO-ISOMORFAS}.

The classification of partial $G$-modules was done in \cite{DLP} and \cite{DN}, there, it was clear that the connected components of the groupoid $\Gamma (G)$ are the key to obtain the simple $\Bbbk_{par}G$-modules. Here, we give an alternative way to describe the simple modules (see Proposition \ref{proposition.Mxalpha.module.structure} and Theorem \ref{simplemodules}). As the category of $\Bbbk_{par}G$-modules is semisimple, one can easily calculate the Grothendieck ring of the monoidal category ${}_{\Bbbk_{par}G}\text{Mod}$. Here, the tensor product is balanced over $A_{par}(\Bbbk G)$, which is the monoidal unit of that category. As the monoidal unit is not simple, one can see that the category ${}_{\Bbbk_{par}G}\text{Mod}$ is a multifusion category in the sense of \cite{EGNObook}.

Finally, the very structure of the partial group algebra as a direct product of matrix algebras with entries in the group algebras of its subgroups suggests that the category of $\Bbbk_{par}G$-modules is closely related to the categories of modules, and the category of partial modules of its subgroups. Then, we have the main theorems of this work. The Christmas Tree's Theorem (Theorem \ref{natal}) states that for any subgroup $H\leq G$, there is a faithful, additive and monoidal functor from the category ${}_{\Bbbk H}\text{Mod}$ to the category ${}_{\Bbbk_{par}G}\text{Mod}$. This means that the partial representations of $G$ encode all representations of its subgroups. The Matryoshka's Theorem (Theorem \ref{matryoshka}) states that for any abelian group $G$ and any subgroup $H\leq G$, there is a faithful, additive and monoidal functor from the category ${}_{\Bbbk_{par}H}\text{Mod}$ to the category ${}_{\Bbbk_{par}G}\text{Mod}$. The motivation for these fancy names came from the images that these results evoked. The $\Bbbk H$-modules of a subgroup $H\leq G$ appeared inside the category of partial modules of $G$ as objects that are written as a tensor product of a $\Bbbk H$ modules by a one dimensional partial $\Bbbk G$ module spanned by the projection over the subgroup $H$. They are tiny objects, similar to glass balls hanging from a Christmas tree. The embedding of the category of partial representations of a subgroup $H\leq G$ was first suggested by the very design of the groupoid graphs associated with these groups. Each connected component of the graph of $\Gamma (H)$ appears disguised in connected components of the graph of $\Gamma (G)$, like a smaller doll nestled inside a larger one.

This article is structured as follows. In Section 2 we discuss the structure of $\Bbbk_{par}G$ and give a brief recap of the main points related to weak Hopf algebras. Section 3 is devoted to prove the isomorphism of the Hopf algebroid structures of $\Bbbk_{par}G$ and $\Bbbk \Gamma (G)$. In Section 4, we give an alternative description of the simple objects in the category of $\Bbbk_{par}G$-modules and describe the monoidal structure of this category. In Section 5, we prove the main theorems of this work: the Christmas Tree's Theorem (Theorem \ref{natal}) and the Matryoshka's Theorem (Theorem \ref{matryoshka}). Finally, in Section 6 we give some perspectives for future works.

\section{Mathematical Preliminaries}
Throughout the text, $\Bbbk$ will denote an algebraically closed field of characteristic zero. 

\subsection{Partial representations of groups and the partial group algebra}

\begin{definition}
\begin{enumerate}
    \item A partial representation of a group $G$ over a unital $\Bbbk$-algebra $B$ is a map $\pi : G \rightarrow B$ such that
    \begin{enumerate}
        \item[(PR1)] $\pi (e)=1_B$, where $e\in G$ is the neutral element of $G$.
        \item[(PR2)] $\pi (g) \pi (h) \pi (h^{-1}) =\pi (gh) \pi (h^{-1})$, for every $g,h\in G$. 
        \item[(PR3)] $\pi (g^{-1}) \pi (g) \pi (h) =\pi (g^{-1}) \pi (gh)$, for every $g,h\in G$.
    \end{enumerate}
    \item A partial module over $G$, or a partial $G$-module, is a pair $(\mathbb{V}, \pi)$, where $\mathbb{V}$ is a $\Bbbk$-vector space and $\pi :G\rightarrow \text{End}_\Bbbk (\mathbb{V})$ is a partial representation of $G$.
    \item A morphism between partial $G$-modules $(\mathbb{V}, \pi )$ and $(\mathbb{W} , \sigma)$ is a $\Bbbk$-linear map $f:\mathbb{V} \rightarrow \mathbb{W}$ such that, for every $g\in G$, $f\circ \pi (g)=\sigma (g)\circ f$.
    \item The category of partial $G$-modules is denoted by ${}_G \mathcal{M}^{par}$.
\end{enumerate}    
\end{definition}

\begin{example}
    Every representation of a group $G$ on a $\Bbbk$-vector space $\mathbb{V}$, that is, a group homomorphism $\pi :G\rightarrow \text{Aut}_\Bbbk (\mathbb{V})$ turns $\mathbb{V}$ into a partial $G$-module.
\end{example}

\begin{example}
    In fact, every partial $G$-module can be obtained out of a global representation of $G$ \cite{Abadiedilations,ABVdilations}. Given a representation $\pi :G\rightarrow \text{Aut}_\Bbbk (\mathbb{V})$, a $\Bbbk$-linear projection $T:\mathbb{V} \rightarrow \mathbb{V}$ is said to satisfy the $c$-condition if for every $g\in G$, we have $T\circ T_g =T_g \circ T$, where $T_g =\pi (g) \circ T \circ \pi (g^{-1})$. The subspace $T(\mathbb{V})\leq \mathbb{V}$ carries a structure of partial $G$-module, $\widetilde{\pi}:G \rightarrow \text{End}_\Bbbk (T(\mathbb{V}))$ given by
    \[
    \widetilde{\pi}(g) (T(v))=T(\pi (g) (T(v))).
    \]
\end{example}

In the same way as $\Bbbk$-linear representations of $G$ can be viewed as modules over the group algebra $\Bbbk G$, we can define a universal algebra, $\Bbbk_{par} G$, such that partial $G$-modules could be viewed as modules over this algebra.

\begin{definition} \label{Partialgroupalgebra} \cite{DEP}
    Given a group $G$, its partial group algebra $\Bbbk_{par} G$ is the algebra 
    \[
    \Bbbk_{par}G = \Bbbk \langle [g]\; | \; g\in G \rangle /\mathcal{I} ,
    \]
    where $\mathcal{I}$ is the ideal generated by relations
    \[
    [e] - 1 , \quad [g][h][h^{-1}]  -  [gh][h^{-1}] ,\quad [g^{-1}][g][h]  -  [g^{-1}][gh] .
    \]
\end{definition}

Note that the map 
\[
\begin{array}{rccc} [ \underline{\;} ]: & G & \rightarrow  & \Bbbk_{par} G\\
\,  & g & \mapsto & [g] \end{array}
\]
is automatically a partial representation of $G$ over the algebra $\Bbbk_{par}G$ and there is the following universal property.

\begin{theorem}  \cite{DEP} For any group $G$ and any partial representation $\pi :G \rightarrow B$, there exists a unique morphism of algebras $\overline{\pi}: \Bbbk_{par}G \rightarrow B$ such that the following diagram commutes

\centerline{\xymatrix{ & \Bbbk_{par}G \ar@{-->}[d]^-{\overline{\pi}} \\
G \ar[ur]^-{[\underline{\;}]} \ar[r]_-{\pi} & B}}  
\end{theorem}

\begin{corollary}
    The category of partial $G$-modules is isomorphic to the category of modules over $\Bbbk_{par}G$.
\end{corollary}

More explicitly, given a partial $G$-module $(\mathbb{V} , \pi)$, we endow $\mathbb{V}$ with a structure of $\Bbbk_{par}G$-module by 
\[
([g_1]\cdots[g_n])\triangleright v=\pi (g_1) \circ \cdots \circ \pi (g_n)(v).
\]
On the other hand, given a $\Bbbk_{par}G$-module $M$, one can define a partial representation $\pi :G\rightarrow \text{End}_{\Bbbk} (M)$ by $\pi (g)(m)=[g]\triangleright m$.

In general, an element of the form $\varepsilon_g =[g][g^{-1}]$ does not coincide with the unit of $\Bbbk_{par}G$, but the subalgebra
\[
A_{par}(\Bbbk G)=\langle \varepsilon_g \; | \; g\in G \rangle
\]
plays an important role in the theory of partial representations of $G$. The main results concerning the subalgebra $A_{par}(\Bbbk G)$ can be summarized in the following proposition.

\begin{proposition} \cite{DEP}
    Let $G$ be a group and consider the subalgebra $A_{par}(\Bbbk G) \subseteq \Bbbk_{par}G$, as defined above. Then
    \begin{enumerate}
        \item For any $g\in G$, we have $\varepsilon_g =\varepsilon_g \, \varepsilon_g$.
        \item For any $g,h \in G$, we have $[g]\varepsilon_h=\varepsilon_{gh}[g]$.
        \item For any $g,h\in G$, we have $\varepsilon_g \, \varepsilon_h =\varepsilon_h \, \varepsilon_g$.
        \item Any element $x=[g_1]\ldots [g_n]\in \Bbbk_{par}G$, can be written as
\[
x=\varepsilon_{g_1}\varepsilon_{g_1g_2}\ldots \varepsilon_{g_1\ldots g_n}[g_1\ldots g_n].
        \]
        \item $A_{par}(\Bbbk G)$ is a partial $G$-module, with the partial representation \\ $\pi :G\rightarrow \text{End}_\Bbbk (A_{par}(\Bbbk G))$ given by
        \[
        \pi (g)(\varepsilon_{h_1}\ldots \varepsilon_{h_n})=\varepsilon_{gh_1} \ldots \varepsilon_{gh_n} \varepsilon_g.
        \]
        For any $a,b\in A_{par}(\Bbbk G)$ and $g,h\in G$, we have 
        \begin{eqnarray*}
           & \ & \pi (g)(ab)=(\pi(g)(a))\, (\pi(g)(b)), \ \textrm{ and } \\ 
          & \ & \pi(g) \circ \pi(h)(a) = \varepsilon_g \pi(gh)(a).            
        \end{eqnarray*}
        
    \end{enumerate}
\end{proposition}

For a finite group $G$, there is a more useful characterization of the algebra $\Bbbk_{par}G$ as a groupoid algebra. Define the sets
\[
\mathcal{P}_{g_1 , \ldots , g_n} (G)=\{ X\subseteq G \; | \; g_1 , \ldots , g_n\in X \} ,
\]
and the set of pairs:
\[
\Gamma (G) =\{ (g, X)\in G\times \mathcal{P}_e (G) \; | \; g^{-1} \in X \} .
\]

\begin{proposition}\label{gammaG}
There is a groupoid structure on $\Gamma (G)$ defined as follows:
\begin{enumerate}
    \item The set of units is $\Gamma (G)^{(0)}=\mathcal{P}_e (G)$.
    \item The source and target maps $s,t: \Gamma (G) \rightarrow \Gamma (G)^{(0)}$ are
    \[
    s(g,X)=X, \qquad t(g,X)=gX=\{ gh\in G \; | \; h\in X \}.
    \]
    \item The unit map $u:\Gamma (G)^{(0)} \rightarrow \Gamma (G)$ is 
    \[
    u(X)=(e,X).
    \]
    \item The multiplication in $\Gamma (G)$ is 
    \[
    (g,X).(h,Y) =\left\{ \begin{array}{lr} (gh, Y) & \text{ if }  X=hY \\
    - &  \text{ otherwise}  \end{array} \right. .
    \]
    \item The inverse of an element $(g,X) \in \Gamma (G)$ is
    \[
    (g,X)^{-1} =(g^{-1} , gX).
    \]
\end{enumerate}
\end{proposition}

\begin{theorem} \label{isomorfismoprincipal}\cite{DEP}
    Let $G$ be a finite group. Then the partial group algebra $\Bbbk_{par}G$ is isomorphic to the groupoid algebra $\Bbbk\Gamma (G)$, with the algebra isomorphism $\lambda :\Bbbk_{par}G \rightarrow \Bbbk \Gamma (G)$ given explicitly by
    \[
    \lambda ([g])=\sum_{X\in \mathcal{P}_{e,g^{-1}}(G)} (g,X),
    \]
    and
    \[
    \lambda^{-1} (g,X)=[g]\prod_{r\in X} \varepsilon_r \prod_{s\notin X} (1-\varepsilon_{s}).
    \]
    If we denote
    \[
    P_X = \prod_{r\in X} \varepsilon_r \prod_{s\notin X} (1-\varepsilon_{s})
    \]
    then we can write 
    \[
    \lambda^{-1} (g,X) = [g]P_X.
    \]
\end{theorem}

\begin{remark} For the rest of this text, we will assume that $G$ is a finite group, even though some results could be proven for any group.
\end{remark}

Note that, if $e\notin X$, the expression $P_X$ automatically vanishes. For future reference, the main properties of the elements $P_X \in A_{par}(\Bbbk G)$ can be summarized as follows.

\begin{proposition}\label{px}{\ }
    \begin{enumerate}
        \item $P_X P_Y =\left\{ \begin{array}{lr} P_X & \text{if } X=Y \\ 0 & \text{ otherwise} \end{array}\right.$, $ \text{for all } X,Y \in \mathcal{P}_e (G)$.
        \item $[g]P_X=\left\{ \begin{array}{lr} P_{gX}[g] & \text{if } g^{-1}\in X \\
        0 & \text{ otherwise} \end{array} \right. $, $\text{for all } X \in \mathcal{P}_e (G)$ and $g\in G$.
        \item $\displaystyle\sum_{X\subseteq G} P_X =\sum_{X\in \mathcal{P}_e (G)} P_X =\mathbbm{1}$.
    \end{enumerate}
\end{proposition}

Introducing the notation for the Boolean value: $[\! [ \text{\bf{sentence}} ]\! ]$ as
\[
[\! [ \text{\bf{sentence}} ]\! ] =\left\{ \begin{array}{lcl} 1 & \text{if \bf{sentence}} & \text{is true} \\
0 & \text{if \bf{sentence}} & \text{is false ,} \end{array}\right.
\]
one can write items (1) and (2) from the Proposition \ref{px} as
\begin{enumerate}
\item $P_X P_Y =[\![ X=Y]\!] P_X$;
\item $[g]P_X= P_{gX}[g] [\! [g^{-1} \in X ]\!]$.
\end{enumerate}

For any $g\in G$, we have
\[
\varepsilon_g=\varepsilon_g \mathbbm{1}=\varepsilon_g \sum_{X\in \mathcal{P}_e(G) } P_X=\sum_{X\in \mathcal{P}_e (G)}P_X [\![ g\in X ]\!]=\sum_{X\in \mathcal{P}_{e,g}(G)}P_X.
\]
Therefore, one can see that the set $\mathcal{B}=\{ P_X \; | \; X\in \mathcal{P}_e(G) \}$ consists of a linear basis for $A_{par}(\Bbbk G)$. Moreover,  for any $X,Y \in \mathcal{P}_e(G)$ define the relation
\[
X\sim Y \quad \Leftrightarrow \quad \exists g\in G, \text{ with } X=gY.
\]
This is easily verified to be an equivalence relation. This implies that there is a subset $T=\{ X_1 , \ldots X_r \} \subseteq \mathcal{P}_e(G)$ such that $X_i \nsim X_j$, for $i\neq j$ and for any $X\in \mathcal{P}_e(G)$, there exists a $X_i \in T$ with $X\sim X_i$. Then, one can decompose the basis of $A_{par}(\Bbbk G)$ as a disjoint union
\[
\mathcal{B} =\dot{\bigcup}_{X_i \in T} \{ P_{gX_i} \; | \; g^{-1} \in X_i \}.
\]

Denoting by $G_X \leq G$ the \textit{isotropy subgroup} of $X\in \mathcal{P}_e(G)$, that is
\[ G_X = \{g \in G \mid \, gX = X\}, \] 
one can define
\[
\Gamma_X =\dfrac{1}{|G_X|}\sum_{g^{-1}\in X} P_{gX} \  = \sum_{Y\sim X} P_Y .
\]
If $X \sim Y$, then it is clear that  $\Gamma_X = \Gamma_Y$. The following result will summarize the main properties of $\Gamma_X$, for any $X \in \mathcal{P}_e(G)$.

\begin{proposition}\label{GammaX}
    \begin{enumerate}
        \item $P_Y \Gamma_X = [\![ X\sim Y]\!] P_Y,$ $\text{for all } X,Y \in \mathcal{P}_e (G)$.
        \item $\Gamma_Y\Gamma_X = [\![ X\sim Y]\!] \Gamma_X$, $\text{for all } X,Y \in \mathcal{P}_e (G)$.
        \item $[g]\Gamma_X = \Gamma_X[g] $, $\text{for all } X \in \mathcal{P}_e (G)$ and $g\in G$.
        \item $ \sum_{X\in T} \Gamma_X =\mathbbm{1}$, where $T=\{ X_1 , \ldots ,X_r \} \subseteq \mathcal{P}_e(G)$ such that $X_i \nsim X_j$, for $i\neq j$ and for any $X\in \mathcal{P}_e(G)$, there exists a $X_i \in T$ with $X\sim X_i$.
    \end{enumerate}
\end{proposition}
\begin{proof} The verification of (1) follows directly from part (1) of Proposition \ref{px} and the definition of $\Gamma_X$, and consequently implies (2). The next two items need more attention. For any $g \in G$ and $X \in \mathcal{P}_e(G)$, we have that 
\begin{eqnarray*}
     [g]\Gamma_X & = &   \dfrac{1}{|G_X|}\sum_{h^{-1}\in X}[g]P_{hX} 
     =   \dfrac{1}{|G_X|}\sum_{h^{-1}\in X}[\![g^{-1}\in hX]\!]P_{ghX}[g] \\
     & = &   \dfrac{1}{|G_X|}\sum_{\tiny{\begin{array}{c} h^{-1}\in X, \\ (gh)^{-1}\in X \end{array}}}P_{ghX}[g]
\end{eqnarray*}
 and on the other hand 
 \begin{eqnarray*}
     & \, & \Gamma_X[g]  =  \displaystyle \dfrac{1}{|G_X|}\sum_{h^{-1}\in X}P_{hX}[g] 
     =  \displaystyle \dfrac{1}{|G_X|}\sum_{h^{-1}\in X}P_{hX}\varepsilon_g[g] \\
      & = & \displaystyle \dfrac{1}{|G_X|}\sum_{h^{-1}\in X}[\![g \in hX]\!]P_{hX}[g]  =  \displaystyle \dfrac{1}{|G_X|}\sum_{\tiny{\begin{array}{c} h^{-1}g\in X,\\ h^{-1} \in X\end{array}}}P_{hX}[g] \\
      & \stackrel{k:=g^{-1}h}{=} & \displaystyle \dfrac{1}{|G_X|}\sum_{\tiny{\begin{array}{c} k^{-1}\in X, \\ (gk)^{-1} \in X \end{array}}}P_{gkX}[g],
 \end{eqnarray*}
 hence $[g]\Gamma_X = \Gamma_X[g] $, $\text{for all } X \in \mathcal{P}_e (G)$ and $g\in G$. The last item follows from the fact that $\displaystyle\sum_{X \in \mathcal{P}_e (G)}P_X = \mathbbm{1}$. Notice that
 $$ \begin{array}{ccl}
     \displaystyle \sum_{X \in T}\Gamma_X 
      & = & \displaystyle \sum_{X\in T}\sum_{Y \sim X}P_{Y} \\
      & = & \displaystyle \sum_{Z\in \mathcal{P}_e(G)} P_{Z} \ = \ \mathbbm{1}
 \end{array} $$
 what proves the last item.
 \qedhere
 
\end{proof}
This last proposition show us that
\[ \kpar G \simeq \bigoplus_{X\in T}\kpar G\, \Gamma_{X}, \]
where $\kpar G\, \Gamma_{X}$ is the ideal generated by $\Gamma_X$.

So if we want to study the algebra $\kpar G$, we only have to study the algebras $\kpar G\, \Gamma_{X}$. We only have to take care of the algebra structure of the algebras $\kpar G\, \Gamma_{X}$, which are algebras with units $\Gamma_X$.

Take $X \in \mathcal{P}_e(G)$, and consider the isotropy group $G_X$ of $X$, then $X$ can be written as a disjoint union of right cosets of $G_X$. More explicitly, there are elements $g_1, ..., g_n \in G$, such that $g_i^{-1} \in X$, for all $i=1,...,n$ and 
\[
X = \dot{\bigcup}_{i=1}^{n} \;  G_Xg_i^{-1}.
\]
Since the union above is disjoint, it is easy to see that if $g_i^{-1}g_j \in G_X$ then $i=j$. Then if $X \sim Y$, there exists an $i \in \{1,..., n\}$ such that $Y=g_iX$ so we can write $\Gamma_X = P_{g_1X} + ... + P_{g_nX}$. 

For any $g \in G$ we have
\begin{eqnarray*}
     & \, & [g]\Gamma_X  =  \displaystyle \sum_{j=1}^n[g]P_{g_jX}  =  \displaystyle \sum_{j=1}^n[g]\varepsilon_{g^{-1}}P_{g_jX} \\
     & = & \displaystyle \sum_{j=1}^n[\![g^{-1}\in g_jX]\!][g]P_{g_jX} =  \displaystyle \sum_{j=1}^n[\![g_j^{-1}g^{-1}\in X]\!][g]P_{g_jX}.
\end{eqnarray*}

Since $ X = \dot{\bigcup}_{i=1}^{n} \;  G_Xg_i^{-1}$, we have
\begin{eqnarray*}
     & \, & [\![g_j^{-1}g^{-1}\in X]\!]  =  \displaystyle \sum_{i=1}^{n} [\![g_j^{-1}g^{-1}\in G_Xg_i^{-1}]\!] \\
     & = & \displaystyle \sum_{i=1}^{n} [\![g_j^{-1}g^{-1}g_i \in G_X]\!]  = \displaystyle \sum_{i=1}^{n} [\![g_i^{-1}gg_j \in G_X]\!]. 
\end{eqnarray*}

Then one conclude that
\begin{equation}\label{equation.gGamma.soma.gijPgjX}
    [g]\Gamma_X = \sum_{i,j = 1}^n [\![g_i^{-1}gg_j \in G_X]\!][g]P_{g_jX}.
\end{equation}
If we denote $g_{ij}:=g_i^{-1}gg_j$, we can rewrite (\ref{equation.gGamma.soma.gijPgjX}) as
\[
[g]\Gamma_X = \sum_{i,j = 1}^n[\![g_{ij} \in G_X]\!][g_ig_{ij}g_j^{-1}]P_{g_jX}.
\]

The previous equation leads us the following result.
\begin{theorem}\label{theorem.kparGGamma.iso.MnKG}
    If $G$ is a finite group, and $X \in \mathcal{P}_e(G)$, then 
    $$ \kpar G\, \Gamma_X \simeq M_{n}(\Bbbk G_X), $$
    as algebras, where $n = \frac{|X|}{|G_X|}$.
\end{theorem}

\begin{proof}
    As we already discussed, there exist $g_1, ..., g_n \in G$ such that
    \begin{center}
        \begin{enumerate}[(1)]
        \item $X = \dot{\bigcup}_{i=1}^{n} \;  G_Xg_i^{-1}$;
        \item[]
        \item $g_i^{-1}g_j \notin G_X, \textrm{ if } i\neq j$;
        \item[] 
        \item $\Gamma_X =\displaystyle \sum_{j=1}^nP_{g_jX}$.
    \end{enumerate}
    \end{center}
    We are going to construct an isomorphism $\varphi:\kpar G\, \Gamma_X \to M_{n}(\Bbbk G_X)$. To do that, consider the following map 
    $$\begin{array}{cccl}
        \pi: & \Bbbk G & \longrightarrow & M_{n}(\Bbbk G_X) \\
         & g & \longmapsto & \displaystyle \sum_{i,j=1}^n\left(\bool{g_i^{-1}gg_j\in G_X}g_i^{-1}gg_j\right)E_{ij}
    \end{array}$$
    where $\{E_{ij}\}_{1\leq i,j \leq n}$ is the standard basis of the $\Bbbk G_X$-free module $M_{n}(\Bbbk G_X)$. Notice that 
    \[
        \pi(e) = \sum_{i,j=1}^n\left(\bool{g_i^{-1}g_j\in G_X}g_i^{-1}g_j\right)E_{ij} = \sum_{i=1}^n e E_{ii} = e I_n.
    \]
    For any $g \in G$, we have
\[
\pi(g)\pi(g^{-1})
=
\sum_{i,j=1}^n
\left(
\sum_{k=1}^n
\bool{g_i^{-1}gg_k\in G_X}
\bool{g_k^{-1}g^{-1}g_j\in G_X}
g_i^{-1}g_j
\right)E_{ij}.
\]
Observe that if
\(
g_i^{-1}gg_k,; g_k^{-1}g^{-1}g_j \in G_X,
\)
then
\(
g_i^{-1}g_j
=
(g_i^{-1}gg_k)(g_k^{-1}g^{-1}g_j)
\in G_X.
\)
Since $G_X$ contains no elements of the form $g_i^{-1}g_j$ for $i\neq j$, it follows that $i=j$. Therefore,
\begin{eqnarray*} & \, & \pi(g)\pi(g^{-1}) = \displaystyle \sum_{i=1}^n\left(\sum_{k=0}^n\bool{g_i^{-1}gg_{k}\in G_X}\bool{g_k^{-1}g^{-1}g_i \in G_X} e\right)E_{ii} \\ & = & \displaystyle \sum_{i=1}^n\left(\sum_{k=0}^n\bool{g_i^{-1}gg_{k}\in G_X}e\right)E_{ii} = \displaystyle \sum_{i=1}^n\left(\sum_{k=0}^n\bool{g_i^{-1}g\in G_Xg_k^{-1}}e\right)E_{ii} \\ & = & \displaystyle \sum_{i=1}^n\left(\bool{g_i^{-1}g\in X}e\right)E_{ii} = \displaystyle \sum_{i=1}^n\left(\bool{g\in g_iX}e\right)E_{ii}. \\ \end{eqnarray*}

Using the formula above, we obtain
\begin{eqnarray*} & \, & \pi(g^{-1})\pi(gh) = \displaystyle \sum_{i,j=1}^n\left(\sum_{k=1}^n\bool{g_i^{-1}g^{-1}g_k\in G_X}\bool{g_k^{-1}ghg_j\in G_X}g_i^{-1}hg_j\right)E_{ij} \\ & = & \displaystyle \sum_{i,j=1}^n\left(\sum_{k=1}^n\bool{g_i^{-1}g^{-1}g_k\in G_X}\right)\bool{g_j^{-1}hg_j\in G_X}g_i^{-1}hg_jE_{ij} \\ & = & \displaystyle \sum_{i,j=1}^n\bool{g_i^{-1}g^{-1}\in \dot{\cup}_{i=1}^{n}G_Xg^{-1}_k}\bool{g_j^{-1}hg_j\in G_X}g_i^{-1}hg_jE_{ij} \\ & = & \displaystyle \sum_{i,j=1}^n\bool{g^{-1}\in g_iX}\bool{g_j^{-1}hg_j\in G_X}g_i^{-1}hg_jE_{ij} = \pi(g^{-1})\pi(g)\pi(h); \\ \end{eqnarray*}
A similar computation shows that
\(
\pi(g)\pi(h)\pi(h^{-1})
=
\pi(gh)\pi(h^{-1}).
\)
Hence, $\pi:\Bbbk G\to M_n(\Bbbk G_X)$ is a partial representation. By the universal property of $\kpar G$, there exist an algebra map $\varphi_X:\kpar G \to M_{n}(\Bbbk G_X)$, such that $\varphi_X([g]) = \pi(g)$, for all $g \in G$. Now notice that, for all $g \in G$, we have 
    \[
    \varphi_X(\varepsilon_g) = \varphi_X([g][g^{-1}]) = \pi(g)\pi(g^{-1}) = \sum_{i=1}^n\left(\bool{g\in g_iX}e\right)E_{ii}.
    \]
    Now for each $i \in \{1, ..., n\}$ notice that
    \[
    \begin{array}{ccl}
        \varphi_X(P_{g_iX}) & = & \displaystyle\prod_{r\in g_iX}\varphi_X(\varepsilon_r)\prod_{s\notin g_iX}\varphi_X(\mathbbm{1}-\varepsilon_s) \\
         & = & \displaystyle \sum_{j=1}^n\left(\prod_{r\in g_iX}\bool{r \in g_jX}\prod_{s\notin g_iX}\left(1-\bool{s\in g_jX}\right)e \right)E_{jj} \\
         & = & \displaystyle \sum_{j=1}^n\left(\bool{g_iX \subseteq g_jX}\prod_{s\notin g_iX}\left(\bool{s\notin g_jX}\right)e \right)E_{jj} \\
         & = & \displaystyle \sum_{j=1}^n\bigl(\bool{g_iX \subseteq g_jX}\bool{g_jX\subseteq g_iX}e \bigr)E_{jj} \ = \ e E_{ii}.
    \end{array}
    \]
    For any $\lambda \in \kpar G$, 
    \begin{equation}\label{equation.varphiX.projection}
        \varphi_X(\lambda) = \varphi_X(\lambda)\sum_{i=1}^nE_{ii} = \varphi_X\left(\lambda\sum_{i=1}^n P_{g_iX}\right) = \varphi_X(\lambda\Gamma_X),
    \end{equation}
    
    hence for $Y \in \mathcal{P}_e(G)$, if $X \nsim Y$ then $\varphi_X(P_Y) = 0$, in fact, 
    \[ \varphi_X(P_Y) = \varphi_X(P_Y\Gamma_X) \stackrel{\textrm{Prop }\ref{GammaX}}{=} 0 .\]
    The Equation (\ref{equation.varphiX.projection}), tell us that $\varphi_X$ can be regarded as a map defined on $\kpar G\, \Gamma_X$, namely,  
    \[
    \begin{array}{cccl}
        \varphi_X: & \kpar G\, \Gamma_X & \longrightarrow & M_n(\Bbbk G_X) \\
         & [g]\Gamma_X & \longmapsto & \displaystyle \pi(g) = \sum_{i,j=1}^n\left(\bool{g_i^{-1}gg_j\in G_X}g_i^{-1}gg_j\right)E_{ij}.
    \end{array}
    \]
    Now we are going to prove that $\varphi_X$ is an isomorphism. For any $i \in \{1,...,n\}$,  $\varphi_X(P_{g_iX}) = E_{ii}$, hence $\varphi_X$ is surjective. Now, to prove that this map is injective, we are going to build a left inverse.
    Consider the following map
    $$
    \begin{array}{rcl}
        \psi_X: \ M_n(\Bbbk G_X) & \longrightarrow & \kpar G\, \Gamma_X, \\
          \displaystyle \sum_{i,j=1}^ng_{ij}E_{ij} & \longmapsto & \displaystyle \sum_{j=1}^n [g_ig_{ij}g_j^{-1}]P_{g_jX}
    \end{array}
    $$
    
    The map $\psi_X$ is indeed a left inverse for $\varphi_X$: take $g \in G$ and notice that 
    \[
    \begin{array}{ccl}
        (\psi_X\circ\varphi_X)([g]\Gamma_X) & = & \displaystyle \psi_X\left(\sum_{i,j=1}^n\left(\bool{g_i^{-1}gg_j\in G_X}g_i^{-1}gg_j\right)E_{ij}\right) \\
         & = & \displaystyle \sum_{j=1}^n\bool{g_i^{-1}gg_j\in G_X}[g_ig_i^{-1}gg_jg_j^{-1}]P_{g_jX} \\
         & = & \displaystyle \sum_{j=1}^n\bool{g_i^{-1}gg_j\in G_X}[g]P_{g_jX} \\
         & \stackrel{(\ref{equation.gGamma.soma.gijPgjX})}{=} & [g]\Gamma_X,
    \end{array}
    \]
    which concludes the proof that $\varphi_X: \kpar G\, \Gamma_X \to M_n(\Bbbk G_X)$ is an isomorphism.   
\end{proof}

Now, let $T\subseteq\mathcal{P}_e(G)$ be a set of representatives of the equivalence classes under $\sim$; that is, $Y\not\sim Z$ for distinct $Y,Z\in T$, and for every $X\in\mathcal{P}_e(G)$, there exists $Y\in T$ such that $X\sim Y$. Then
    \[
    \kpar G \simeq \bigoplus_{X\in T} \kpar G\, \Gamma_X \simeq \bigoplus_{X\in T} M_{n_X}(\Bbbk G_X),
    \]
    where $n_X = \frac{|X|}{|G_X|}$.
    
    \begin{remark}{ The isomorphism above is not new; it was already established in \cite{DEP}. However, the construction presented here is more suitable for our purposes. Indeed, it makes explicit the role of the idempotents $P_X$ and $\Gamma_X$ in the decomposition of $\kpar G$, and the isomorphism
\(
\varphi_X:\kpar G\, \Gamma_X\to M_{n_X}(\Bbbk G_X)
\)
is constructed explicitly in terms of these same idempotents. This makes the correspondence between the $\Bbbk G_X$-modules and the $\kpar G$-modules more transparent, while keeping track of the idempotent decomposition of $\kpar G$.}
\end{remark}

\subsection{Weak Hopf algebras}
\begin{definition} \cite{BNSweak}\label{definicaobifraca}
	A weak bialgebra is a quintuple  $(H, \mu ,u , \Delta, \epsilon )$ in which:
	\begin{itemize}
		\item[$(1)$] $(H, \mu , u)$ is a $\Bbbk$-algebra.
		\item[$(2)$] $(H, \Delta , \epsilon)$ is a $\Bbbk$-coalgebra.
		\item[$(3)$] $\Delta(xy)=\Delta(x)\Delta(y)$, for every $x,y$ in $H$.
            \item[$(4)$]$\epsilon (xyz) = \epsilon(xy_{(1)}) \epsilon(y_{(2)}z) = \epsilon(xy_{(2)})\epsilon(y_{(1)}z) $, for every $x,y,z$ in $H$.
		\item[$(5)$] $(\Delta(1) \otimes 1)(1 \otimes \Delta(1)) = 1_{(1)} \otimes 1_{(2)} \otimes 1_{(3)} = (1 \otimes \Delta (1))(\Delta(1) \otimes 1) $. Using the Sweedler notation, one can rewrite this condition as 
\[ 1_{(1)} \otimes 1_{(2)} 1_{(1')}\otimes 1_{(2')} = 1_{(1)} \otimes 1_{(2)} \otimes 1_{(3)} = 1_{(1)} \otimes 1_{(1')} 1_{(2)} \otimes 1_{(2')}.\]
	\end{itemize}
\end{definition}
Observe that every bialgebra is automatically a weak bialgebra. In fact, a weak bialgebra is a bialgebra if, and only if $\Delta (1)=1\otimes 1$, or equivalently $\epsilon (xy)=\epsilon (x) \epsilon (y)$, for all $x,y \in H$

\begin{definition}\cite{BNSweak}\label{subalgebra-Ht-Hs} Given a weak bialgebra $H$, we define linear maps $\epsilon_s ,\epsilon_t, \epsilon_s', \epsilon_t':H\rightarrow H$ given, respectively, by
\begin{eqnarray*}
\epsilon_s (h) =1_{(1)}\epsilon(h1_{(2)}) \quad & , & \quad \epsilon_t (h)= \epsilon(1_{(1)}h)1_{(2)} \\
\epsilon_s' (h) =1_{(1)}\epsilon(1_{(2)}h) \quad & , & \quad \epsilon_t' (x)=\epsilon(h1_{(1)})1_{(2)}.
\end{eqnarray*}
\end{definition}
These maps define two subspaces of $H$, namely $H_s:=Im(\epsilon_s)=Im(\epsilon_s')$ and $H_t := Im(\epsilon_t)=Im (\epsilon_t')$. 

These are idempotent maps in $\text{End}(H)$ with respect to the composition and preserve the unit $1_H$. One can easily see that $\Delta(1) \in H_s \otimes H_t$ and that the subspaces $H_s$ and $H_t $ are mutually commutative subalgebras of $H$, that is, for any $x\in H_s$ and $y\in H_t$, we have $xy=yx$. Moreover one can prove that the subalgebras $H_s$ and $H_t$ are Frobenius separable algebras over the base field $\Bbbk$, in particular, it implies that $H_s$ and $H_t$ are finite dimensional (see, for example, \cite{BWcorings} Theorem 36.8).

\begin{definition} \cite{BNSweak}\label{defhopffraca}
A  weak Hopf algebra is a sextuple $(H, \mu , u, \Delta, \epsilon, S)$, in which $(H, \mu , u, \Delta, \epsilon)$ is a weak bialgebra and  $S:H \rightarrow H$ is a $\Bbbk$-linear map satisfying
\begin{itemize}
	\item [$i)$] $\epsilon_s(h)= S(h_{(1)})h_{(2)}$;
	\item [$ii)$] $\epsilon_t(h ) = h_{(1)} S(h_{(2)}) $;
	\item[$iii)$] $S(h)= S(h_{(1)})h_{(2)} S(h_{(3)}) $.
\end{itemize}
\end{definition}
The map $S$ is called antipode and one can verify its uniqueness.

There is another useful characterization of Weak Hopf algebras given by the following result:

\begin{theorem}\label{Teorema-Schauenburg}\cite{Schauenburg}
Let $H$ be a weak bialgebra, then $H$ is a weak Hopf algebra if, and only if the canonical map 
$$\begin{matrix}
	can: H \otimes_{H_s} H &\longrightarrow& H \otimes H \\
	\ \  \  h \otimes k &  \longmapsto& \ \ h_{(1)} \otimes h_{(2)}k,
\end{matrix}$$
induces the isomorphism $H\otimes_{H_s} H \cong \Delta (1)  (H \otimes H)=:H\boxtimes H$.
\end{theorem}

\begin{example} Every Hopf algebra is a weak Hopf algebra. In fact, a weak Hopf algebra $H$ is a Hopf algebra if, and only if $\Delta (1)=1\otimes 1$, or, equivalently, if and only if for every $h,k\in H$, $\epsilon (hk)=\epsilon (h) \epsilon (k)$.
\end{example}

For our purposes, the paradigmatic example of a weak Hopf algebra is the algebra of a finite groupoid.

\begin{example}\label{exemplo-kgama-bialgebrafraca}
    Let $\Gamma$ be a groupoid with finitely many objects, then its groupoid algebra is 
    \[
    \Bbbk \Gamma =\{ \sum_{\gamma \in \Gamma} a_{\gamma} \gamma \; | \; a_\gamma \in \Bbbk, \quad \textrm{finite sums} \}. 
    \]
    Note that the elements $\gamma \in \Gamma$ forms a basis for the vector space $\Bbbk \Gamma$. The algebra structure is given, for the basis elements, by the multiplication of the groupoid 
    \[
    \gamma.\delta=\left\{ \begin{array}{lcr} \gamma \delta, & \text{if} &(\gamma ,\delta) \in \Gamma^{(2)}, \\
    0,& & \text{otherwise,}\end{array} \right.
    \]
    and extended linearly. The unit is $1_{\Bbbk \Gamma} =\sum_{x\in \Gamma^{(0)}} x$. For the sake of simplicity, we are using the identification $\Gamma^{(0)}=u(\Gamma^{(0)})\subseteq \Gamma$. {Concerning the coalgebra structure, the comultiplication, counit are,  respectively, given by $\Delta (\gamma)=\gamma \otimes \gamma$,  $\epsilon(\gamma)=1_{\Bbbk}$. The antipode writes $S(\gamma)=\gamma^{-1}$, for all $\gamma \in \Bbbk \Gamma$. With these operations, $\Bbbk \Gamma$ is a weak Hopf algebra.} From these definitions, the comultiplication of the unit results in
    \[
    \Delta (1_{\Bbbk \Gamma})=\sum_{x\in \Gamma^{(0)}} x\otimes x.
    \]
    In this situation, we have, for any $\gamma \in \Gamma$
    \[
    \epsilon_s (\gamma) =\sum_{x\in \Gamma^{(0)}} x\epsilon (\gamma x)=s(\gamma)=\epsilon_t' (\gamma) ,
    \]
    and
    \[
    \epsilon_t (\gamma) =\sum_{x\in \Gamma^{(0)}} \epsilon (x\gamma )x=t(\gamma)=\epsilon_s' (\gamma) .
    \]
    As $H=\Bbbk \Gamma$ is cocommutative, the subalgebras $H_s$ and $H_t$ coincide with the commutative algebra $\Bbbk \Gamma^{(0)}$. 

    In particular, for a finite group $G$, the groupoid algebra $\Bbbk \Gamma (G)$ (see Proposition \ref{gammaG}) is a weak Hopf algebra.
\end{example}

\section{The isomorphism of Hopf algebroids}

The isomorphism $\lambda : \Bbbk_{par}G\rightarrow \Bbbk\Gamma(G)$ in Theorem \ref{isomorfismoprincipal} goes beyond the simple isomorphism of algebras. In fact, we can prove that $\lambda$ is an isomorphism of Hopf algebroids. Let us briefly recall the definitions of a bialgebroid and of a Hopf algebroid.

\begin{definition}\cite{Bohm}
	Given a $\Bbbk$-algebra $A$, a left (resp. right) bialgebroid over $A$ is a  sextuple  $(\mathcal{H}, A, s_l, t_l, \underline{ \Delta }_l , {\underline{ \epsilon }}_l)$ (resp.  $ (\mathcal{H}, A, s_r, t_r, \underline{ \Delta }_r , \underline{ \epsilon }_r)$) in which:
	\begin{itemize}
		\item[$ (1) $] $\mathcal{H}$ is a $\Bbbk$-algebra.
		\item[$ (2) $] The map $s_l:A\rightarrow \mathcal{H}$ (resp. $s_r: A\rightarrow \mathcal{H}$) is an algebra morphism. The map $t_l:A\rightarrow \mathcal{H}$ (resp. $t_r:A\rightarrow \mathcal{H}$) is an antimorphism of algebras. Moreover for every  $a, b \in A$, we have $s_l(a)t_l(b)=t_l(b)s_l(a)$ (resp.  $s_r(a) t_r(b)=t_r(b)s_r(a)$). The algebra $\mathcal{H}$ is endowed with a bimodule structure over $A$ by the maps $s_l,t_l$ (resp. $s_r, t_r$) by $a \triangleright h \triangleleft b = s_l(a)t_l(b)h$ (resp. $a \triangleright h \triangleleft b = hs_r (b) t_r(a)$), 
  The maps $s_l$ and $t_l$ are usually called (left) source and (left) target, respectively, the maps $s_r$ and $t_r$ are called (right) source and (right) target. 
		\item[$ (3) $] The triple $(\mathcal{H} , \underline{ \Delta }_l , \underline{ \epsilon }_l )$ (respect. $( \mathcal{H}, \underline{ \Delta }_r , \underline{ \epsilon }_r )$) is an $A$-coring relative to the structure of $A$-bimodule given by $s_l$ and $t_l$ (resp. $s_r$ and $t_r$).
		\item[$ (4) $] The comultiplication $\underline{ \Delta }_l$ (respect. $\underline{ \Delta }_r$) has its image in the Takeuchi algebra
		$$\mathcal{H} {\times }_A \ \mathcal{H} = \big\{ \sum_{i} h_i \otimes k_i \in \mathcal{H}\otimes_{A} \mathcal{H} : \sum_i h_i t_l(a) \otimes k_i=\sum_i h_i \otimes k_i s_l(a) \ \forall a \in A \big\}$$
		(resp.
		$$ \mathcal{H}\ { }_A\times \mathcal{H} = \big\{ \sum_{i} h_i \otimes k_i \in \mathcal{H}\otimes_{A} \mathcal{H} : \sum_i s_r (a) h_i  \otimes k_i=\sum_i h_i \otimes t_r(a) k_i  \ \forall a \in A \big\} ),$$
		and it is a morphism of algebras.
		\item[$ (5) $] For every $h,k \in \mathcal{H}$, we have
		$$\underline{ \epsilon }_l (hk) = \underline{ \epsilon }_l( h s_l(\underline{ \epsilon}_l (k) ) )= \underline{ \epsilon }_l ( ht_l( \underline{ \epsilon }_l (k)))$$
		\big(resp.
		$$ \underline{ \epsilon }_r (hk) = \underline{ \epsilon }_r ( s_r (\underline{ \epsilon}_r(h) )k )= \underline{ \epsilon }_r ( t_r( \underline{ \epsilon }_r (h))k)\ \big).$$
		
	\end{itemize}
\end{definition} 

\begin{definition}\cite{Bohm}\label{def-morfismo-HOPF-ALGEBROIDE}
    Let $\mathcal{H}_l=(\mathcal{H}, A, s, t, \underline{ \Delta }_l , {\underline{ \epsilon }}_l)$ and $\mathcal{H}_l'=(\mathcal{H}', A', s', t', \underline{ \Delta }_l' , {\underline{ \epsilon }}_l')$ be two left $A$-bialgebroids. A morphism of left $A$-bialgebroids between $\mathcal{H}$ and $\mathcal{H}'$ is a pair of linear maps $$\varPhi: \mathcal{H} \longrightarrow \mathcal{H}', \ \ \  \phi: A \longrightarrow A'$$
    such that,
    \begin{align*}
        s' \circ \phi &= \varPhi \circ s  \\
        t' \circ \phi &=  \varPhi \circ t \\
        \underline{\epsilon}_l' \circ \varPhi &= \phi \circ \underline{\epsilon}_l  \\
        \underline{\Delta}_l ' \circ \varPhi &= (\varPhi \otimes \varPhi) \circ \underline{\Delta}_l .
    \end{align*}
\end{definition}

The definition of a morphism of right $A$-bialgebroids is completely analogous.

\begin{definition}\cite{Bohm}\label{definicao-HOPF-ALGEBROIDE}
	Consider two algebras $A$ and $\tilde{A}$ which are antiisomorphic, that is, $A \cong \tilde{A}^{op}$. Consider also a $\Bbbk$-algebra $\mathcal{H}$ such that $\mathcal{H}_l:=(\mathcal{H}, A, s_l, t_l , \underline{\Delta}_l , \underline{\epsilon}_l)$ is a left $A$-bialgebroid and  $\mathcal{H}_r:=(\mathcal{H}, \tilde{A}, s_r, t_r , \underline{\Delta}_r , \underline{\epsilon}_r)$ is a right $\tilde{A}$-bialgebroid. We define a Hopf algebroid structure on $\mathcal{H}$ as a triple $(\mathcal{H}_l,\mathcal{H}_r, \mathcal{S})$, in which $\mathcal{S}:\mathcal{H} \longrightarrow \mathcal{H}$ is an algebra antimorphism such that, 

\begin{itemize}
	\item[$(i)$] $s_l \circ \underline{\epsilon}_l \circ t_r= t_r$, $t_l \circ \underline{ \epsilon}_l \circ s_r=s_r$, $s_r \circ \underline{ \epsilon}_r \circ t_l = t_l$, $t_r\circ \underline{ \epsilon }_r \circ s_l = s_l$;
	\item[$(ii)$] $(\underline{ \Delta}_l \otimes_{\tilde{A}} \mathcal{H} )\circ \underline{ \Delta}_r = ( \mathcal{H} \otimes _{A} \underline{ \Delta}_r )\circ \underline{\Delta}_l$, \ $( \mathcal{H} \otimes _{ \tilde{A} } \underline{ \Delta }_l )\circ \underline{ \Delta}_r = (\underline{ \Delta}_r \otimes _{A} \mathcal{H}) \circ \underline{\Delta}_l$;
	\item[$(iii)$] $\mathcal{S}(t_l(a) h t_r(b' ) )= s_r(b') \mathcal{S}(h) s_l(a) $, for every $a\in A$, $b'\in \tilde{A}$ and $h\in \mathcal{H}$;
	\item[$(iv)$] $\mu_{\mathcal{H}} \circ (\mathcal{S} \otimes \mathcal{H})\circ \underline{\Delta}_{\mathit{l}} = s_r \circ \underline{\epsilon}_r $, and \ $\mu_{\mathcal{H}} \circ (\mathcal{H} \otimes \mathcal{S})\circ \underline{\Delta}_r = s_l\circ \underline{\epsilon}_l$.
\end{itemize}
\end{definition}

\begin{definition} \label{strictmorphism} \cite{BS}
A morphism of two Hopf algebroids 
\[(\mathcal{H}, A, \tilde{A}, s_l,t_l, s_r , t_r, \Delta_l, \Delta_r , \epsilon_l , \epsilon_r , \mathcal{S})
\] 
and
\[
(\mathcal{H}', A', \tilde{A}', s'_l,t'_l ,s'_r , t'_r, \Delta'_l, \Delta'_r , \epsilon'_l , \epsilon'_r , \mathcal{S}')
\] 
is a morphism of left bialgebroids $(\Phi , \phi )$  between them. 

    A Hopf algebroid morphism is called strict if $\Phi \circ \mathcal{S} =\mathcal{S}' \circ \Phi$.
\end{definition}

Given a Hopf algebroid $(\mathcal{H}, A, \tilde{A}, s_l,t_l ,s_r , t_r, \Delta_l, \Delta_r , \epsilon_l , \epsilon_r , \mathcal{S})$, a left $\mathcal{H}$-module $M$ naturally inherits the structure of an $A$-bimodule by
\[
a\triangleright m\triangleleft b=s_l (a)t_l (b)\cdot m.
\]
Similarly, a right $\mathcal{H}$-module $N$ inherits a structure of an $\tilde{A}$-bimodule by
\[
\tilde{a}\blacktriangleright n\blacktriangleleft \tilde{b} =n\cdot s_r (\tilde{b})t_r (\tilde{a}).
\]
Moreover, we have the following result.

\begin{theorem}
    Let $(\mathcal{H}, A, \tilde{A}, s_l,t_l , s_r , t_r, \Delta_l, \Delta_r , \epsilon_l , \epsilon_r , \mathcal{S})$ be a Hopf algebroid. Then the category ${}_{\mathcal{H}}\text{Mod}$ of left $\mathcal{H}$-modules has a structure of a closed monoidal structure such that the forgetful functor $U:{}_{\mathcal{H}}\text{Mod} \rightarrow {}_A \text{Mod}_A$ is strictly monoidal.
\end{theorem}

\begin{example}
    Let $A$ be an algebra. One can endow the algebra  $A\otimes A^{op}$ with a structure of Hopf algebroid $(A\otimes A^{op} , A, A^{op}, s_l, t_l, s_r, t_r, \Delta_l , \Delta_r , \epsilon_l, \epsilon_r , \mathcal{S})$ in which 
    \begin{itemize}
        \item $s_l:A \longrightarrow A\otimes A^{op}$ is given by $s_l(b)=b\otimes 1$;
        \item $t_l: A \longrightarrow A\otimes A^{op}$ is given by $t_l(b)=1\otimes b$;
        \item $s_r=t_l$ and $t_r=s_l$;
        \item $\Delta_l:A\otimes A^{op} \longrightarrow (A\otimes A^{op}) \otimes _ A (A\otimes A^{op})$ is given by $\Delta_l (a\otimes b)=(a\otimes 1)\otimes_A (1\otimes b)$;
        \item $\Delta_r:A\otimes A^{op} \longrightarrow (A\otimes A^{op}) \otimes _ {A^{op}} (A\otimes A^{op})$ is given by $\Delta_l (a\otimes b)=(a\otimes 1)\otimes_{A^{op}} (1\otimes b)$;
        \item $\epsilon_l:A\otimes A^{op}\rightarrow A$ is given by $\epsilon_l(a\otimes b)=ab$;
        \item $\epsilon_r:A\otimes A^{op}\rightarrow  A^{op}$ is given by $\epsilon_r(a\otimes b)=ab$;
        \item $\mathcal{S}: A\otimes A^{op} \longrightarrow A\otimes A^{op}$ is defined by $\mathcal{S}(a\otimes b)= b\otimes a$.
    \end{itemize} 
\end{example}

\begin{example}
    Every weak Hopf algebra $H$ is a Hopf algebroid. Indeed, $H$ has a left bialgebroid structure over the subalgebra $H_t =\epsilon_t (H)$. The left bialgebroid $(H,H_t , s_l, t_l, \Delta_l , \epsilon_l)$ is defined by:
    \begin{itemize}
        \item $s_l =\imath :H_t \rightarrow H$ is the canonical inclusion.
        \item $t_l =\epsilon_s'|_{H_t}:H_t \rightarrow H$, as defined in Definition \ref{subalgebra-Ht-Hs}.
        \item $\Delta_l :\pi_t \circ \Delta :H\rightarrow H\otimes_{H_t} H$, where $\pi_t :H\otimes H \rightarrow H\otimes_{H_t}H$ is the canonical projection.
        \item $\epsilon_l =\epsilon_t$, as defined in Definition \ref{subalgebra-Ht-Hs}.
    \end{itemize}
    Similarly, there is a right bialgebroid structure $(H,H_s , s_r, t_r, \Delta_r , \epsilon_r)$, where $H_s=\epsilon_s (H)$, given by
    \begin{itemize}
        \item $s_r =\imath :H_s \rightarrow H$ is the canonical inclusion.
        \item $t_r =\epsilon_t'|_{H_s}:H_s \rightarrow H$, as defined in Definition \ref{subalgebra-Ht-Hs}.
        \item $\Delta_r :\pi_s \circ \Delta :H\rightarrow H\otimes_{H_s} H$, where $\pi_s :H\otimes H \rightarrow H\otimes_{H_s}H$ is the canonical projection.
        \item $\epsilon_l =\epsilon_s$, as defined in Definition \ref{subalgebra-Ht-Hs}.
    \end{itemize}
    Finally, the antipode $S:H\rightarrow H$, from Definition \ref{defhopffraca} is the antipode of $H$ as a Hopf algebroid.
\end{example}

\begin{example}
    In particular, for a groupoid $\Gamma$ with finitely many objects, the groupoid algebra $\Bbbk\Gamma$ presented in Example \ref{exemplo-kgama-bialgebrafraca} is a weak Hopf algebra, therefore it is also a Hopf algebroid over $A=\Bbbk \Gamma^{(0)}=(\Bbbk\Gamma)_t=(\Bbbk\Gamma)_s$, using again the identification $\Gamma^{(0)}=u(\Gamma^{(0)})\subseteq \Gamma$ in the groupoid. In this case, we have, for any $x\in \Gamma^{(0)}$,
    \[
    \epsilon_s (x)=\epsilon_t (x) =\epsilon_s' (x)=\epsilon_t' (x)=x,
    \]
    that is, the left and right source and target maps coincide with the canonical inclusion $\imath: A \rightarrow \Bbbk \Gamma$. Moreover
    \begin{eqnarray*}
    & \, & \Delta_l (\gamma) =\Delta_r (\gamma) =\gamma \otimes_A \gamma ;\\
    & \, & \epsilon_l (\gamma )=t(\gamma) ;\\
    &\, & \epsilon_r (\gamma)=s(\gamma) ; \\
    &\, & \mathcal{S}(\gamma)= \gamma^{-1} .
    \end{eqnarray*}
    Note that, for any $\gamma\otimes_{A} \delta \neq 0$ in $\Bbbk\Gamma \otimes_{A} \Bbbk\Gamma$ and $x\in \Gamma^{(0)}$
    \[
     \gamma\triangleleft x \otimes_{A} \delta =  \gamma \otimes_A x\triangleright \delta ,
    \]
    that is
    \[
    x\gamma \otimes_A \delta=\gamma \otimes_A x\delta .
    \]
     As $\gamma \otimes_A \delta \neq 0$, this forces the target of $\gamma$ being equal to the target of $\delta$ in the groupoid. In particular
     for any $\gamma \in \Gamma$, that
    \begin{equation}\label{coprodutonogrupoide}
    \Delta_l(\gamma)= \Delta_r (\gamma)= \sum_{x\in \Gamma^{(0)}} x\gamma \otimes_A x\gamma .
    \end{equation}
    \end{example}

\begin{example} \cite{ABVparreps}
    For a group $G$ the partial group algebra $\Bbbk_{par}G$, from Definition \ref{Partialgroupalgebra}, is also a Hopf algebroid over the commutative base algebra $A_{par}(\Bbbk G)=\langle \varepsilon_g =[g][g^{-1}] \; | \; g\in G \rangle$. This comes from the construction of the ``partial Hopf algebra'' $H_{par}$, in \cite{ABVparreps}, which factorizes partial representations of a Hopf algebra $H$ by morphisms of algebras. In our particular case, for the group algebra, $\Bbbk G$, we have $\Bbbk_{par}G=(\Bbbk G)_{par}$. The structure of Hopf algebroid of $\Bbbk_{par}G$ is given by
    \begin{eqnarray*}
        & \, & s_l (\varepsilon_g) =  s_r (\varepsilon_g) =t_l (\varepsilon_g)=t_r(\varepsilon_g) =[g][g^{-1}] ;\\
        & \, & \Delta_l ([g_1]\cdots[g_n])  =  \Delta_r ([g_1]\cdots [g_n])=[g_1]\cdots [g_n]\otimes_{A_{par}(\Bbbk G)} [g_1]\cdots [g_n] ;\\
        & \, & \epsilon_l ([g_1]\cdots [g_n])  =  \epsilon_r ([g_1]\cdots [g_n])=\varepsilon_{g_1} \varepsilon_{g_1 g_2} \cdots \varepsilon_{g_1 \ldots g_n} ;\\
        & \, & \mathcal{S}([g_1]\cdots [g_n])  =  [g_n^{-1}]\cdots [g_1^{-1}].
    \end{eqnarray*}
\end{example}

\begin{theorem}\label{kpar-e-kgamma-SAO-ISOMORFAS}
    The algebra isomorphism $\lambda :\Bbbk_{par}G \rightarrow \Bbbk \Gamma (G)$, given by
    \[
    \lambda ([g])=\sum_{X\in \mathcal{P}_{e,g^{-1}}} (g,X),
    \]
    induces an isomorphism of Hopf algebroids between $\Bbbk_{par} G$ and $\Bbbk \Gamma (G)$.
\end{theorem}

\begin{proof}
First recall that the expression of the inverse map $\lambda^{-1}:\Bbbk \Gamma \rightarrow \Bbbk_{par}G$ is
\[
\lambda^{-1}\left( (g,X) \right) =[g]P_X.
\]
In particular, $\lambda$ gives an algebra isomorphism between $A_{par}(\Bbbk G)$ and the subalgebra 
\[
A_{\Gamma}=\Bbbk \Gamma (G)^{(0)} =\langle (e,X) \; | \; X\in \mathcal{P}_e (G) \rangle \subseteq \Bbbk\Gamma (G).
\]
Also, recall that $A_{par}(\Bbbk G)\subseteq \Bbbk_{par}G$, and then $s_l (\varepsilon_g)=t_l(\varepsilon_g)=s_r(\varepsilon_g)=t_r(\varepsilon(g))=[g][g^{-1}]$, that is, the left and right source and target maps are just the inclusion map. 
Checking the values of $\lambda$ on the basis elements of $A_{par}(\Bbbk G)$, $P_X=s_l(P_X)=s_r(P_X)=t_l(P_X)=t_r(P_X) \in \Bbbk_{par}G$, we have
\[
\lambda (s_l(P_X))=\lambda (P_X)=\lambda (\lambda^{-1}(e,X))=(e,X) =s_l ((e,X)).
\]
The same is valid for $s_r$, $t_l$ and $t_r$.

For the comultiplication, on the one hand we have
\[
    \underline{\Delta}_l (\lambda ([g])  =  \underline{\Delta}_l \left( \sum_{X\in \mathcal{P}_{e,g^{-1}}(G)} (g,X) \right)
    =  \sum_{X\in \mathcal{P}_{e,g^{-1}}(G)} (g,X) \otimes_{A_\Gamma} (g,X) .
\]

Now, consider the map
\[
\begin{array}{rccc} \Lambda :& \Bbbk_{par}G \times \Bbbk_{par}G & \rightarrow &  \Bbbk \Gamma (G) \otimes_{A_\Gamma} \Bbbk \Gamma (G) \\
\, & (\xi , \zeta ) & \mapsto & \lambda (\xi ) \otimes_{A_\Gamma} \lambda (\zeta) \end{array} .
\]
It is easy to see that this map $\Lambda$ is balanced over $A_{par}(\Bbbk G)$, so there is a unique well defined linear map 
\[
(\lambda \otimes_{A_{par}(\Bbbk G)}\lambda) :\Bbbk_{par}G \otimes_{A_{par}(\Bbbk G)} \Bbbk_{par}G \rightarrow \Gamma (G) \otimes_{A_\Gamma} \Bbbk \Gamma (G)
\]
factoring $\Lambda$. Then we have,
\begin{align*}
    & (\lambda \otimes_{A_{par}(\Bbbk G)}\lambda)\circ \Delta_l ([g]) =\lambda ([g]) \otimes_{A_\Gamma}\lambda ([g]) \\
    & =  \sum_{X,Y\in \mathcal{P}_{e,g^{-1}}(G)} (g,X) \otimes_{A_\Gamma} (g,Y) =\\
    & = \sum_{Z\in \mathcal{P}_e(G)} \ \sum_{X,Y\in \mathcal{P}_{e,g^{-1}}(G)} (e,Z) (g,X) \otimes (e,Z)(g,Y) \\
    & =  \sum_{X\in \mathcal{P}_{e,g^{-1}}(G)} (g,X) \otimes_{A_\Gamma} (g,X) .
\end{align*}

For the counit, observe that
\begin{align*}
& \lambda|_{A_{par}(\Bbbk G)} (\epsilon_l ([g_1]\cdots[g_n]))  =  \lambda|_{A_{par}(\Bbbk G)} (\varepsilon_{g_1}\cdots \varepsilon_{g_1\cdots g_n}) = \\
& =  \left( \sum_{X_1\in \mathcal{P}_{e,g_1}(G)} (e,X_1)\right)  \cdots \left( \sum_{X_n\in \mathcal{P}_{e,g_1\cdots g_n}(G)} (e,X_n)\right) \\
& =  \sum_{X_1 \in \mathcal{P}_{e,g_1}(G) }  \cdots \sum_{X_n\in \mathcal{P}_{e,g_1\cdots g_n}(G)} (e,X_n) 
[\![ X_1 =X_2 ]\!] \cdots [\![ X_{n-1} =X_n ]\!] \\
& =  \sum_{X\in \mathcal{P}_{e,g_1, g_1 g_2 \ldots, g_1\cdots g_n}(G)} (e,X)
\end{align*}
On the other hand
\begin{align*}
&  \underline{\epsilon}_l (\lambda([g_1]\cdots [g_n]) =  \underline{\epsilon}_l \left(\left( \sum_{Y_1 \in \mathcal{P}_{e, g_1^{-1}}(G)} (g_1 ,Y_1) \right) \cdots \left( \sum_{Y_n \in \mathcal{P}_{e, g_n^{-1}}(G)} (g_n ,Y_n) \right) \right)\\
& =  \underline{\epsilon}_l\left( \sum_{Y_1 \in \mathcal{P}_{e, g_1^{-1}}(G)}  \cdots \sum_{Y_n \in \mathcal{P}_{e, g_n^{-1}}(G)} (g_1 \cdots g_n ,Y_n) [\![ Y_1 =g_2 Y_2 ]\!] \cdots [\![ Y_{n-1} =g_n Y_n ]\!] \right)\\
& =  \underline{\epsilon}_l\left(  \sum_{Y \in \mathcal{P}_{e, g_n^{-1}, g_n^{-1}g_{n-1}^{-1}, \ldots , g_n^{-1}\cdots g_1^{-1}}(G)}  (g_1 \cdots g_n ,Y) \right) \\
& =  \sum_{Y \in \mathcal{P}_{e, g_n^{-1}, g_n^{-1}g_{n-1}^{-1}, \ldots , g_n^{-1}\cdots g_1^{-1}}(G)}  (e ,g_1 \cdots g_n Y)  =  \sum_{X\in \mathcal{P}_{e,g_1, g_1 g_2 \ldots , g_1\cdots g_n}(G)} (e,X)
\end{align*}
Last, but not least, the antipode is also preserved by $\lambda$. Indeed

\begin{eqnarray*}
& \, &\lambda (\mathcal{S} ([g_1]\cdots [g_n])  = \lambda ([g_n^{-1}] \cdots [g_1^{-1}]) \\
& = & \left( \sum_{Y_n \in \mathcal{P}_{e, g_n}(G)} (g_n^{-1} , Y_n ) \right) \cdots \left( \sum_{Y_1 \in \mathcal{P}_{e, g_1}(G)} (g_1^{-1} , Y_1 ) \right) \\
& = & \left( \sum_{X_n \in \mathcal{P}_{e, g_n^{-1}}(G)} (g_n^{-1} , g_n X_n ) \right) \cdots \left( \sum_{X_1 \in \mathcal{P}_{e, g_1^{-1}}(G)} (g_1^{-1} , g_1 X_1 ) \right) \\
& = & \left( \sum_{X_n \in \mathcal{P}_{e, g_n^{-1}}(G)} \mathcal{S} (g_n , X_n ) \right) \cdots \left( \sum_{X_1 \in \mathcal{P}_{e, g_1^{-1}}(G)} \mathcal{S}(g_1 , X_1 ) \right) \\
& = & \mathcal{S} \left( \left( \sum_{X_1 \in \mathcal{P}_{e, g_1^{-1}}(G)} (g_1 , X_1 ) \right) \cdots \left( \sum_{X_n \in \mathcal{P}_{e, g_n^{-1}}(G)} (g_n , X_n ) \right) \right)\\
& = & \mathcal{S} \left( \lambda ([g_1]\cdots [g_n]) \right) 
\end{eqnarray*}

Therefore, the map $\lambda$ is a strict isomorphism of Hopf algebroids between $\Bbbk_{par}G$ and $\Bbbk \Gamma (G)$, according to Definition \ref{strictmorphism}.
\end{proof}

\section{The monoidal category of partial representations of a finite group}

As we have seen before, the algebra $\Bbbk_{par}G$ can be expressed as a direct sum \cite{DEP, DP}
\[
\Bbbk_{par}G =\bigoplus_{\begin{array}{c} H\leq G \\
1\leq m\leq [G:H]\end{array}} c_m (H)M_m (\Bbbk H),
\]
with $c_m (H)\in \mathbb{Z}_+$ being the multiplicity of the corresponding summand \cite{DP}. As $\Bbbk$ is an algebraically closed field of characteristic $0$, each group algebra of each subgroup is also decomposed as a product of matrix algebras. Therefore, as a consequence of the Wedderburn-Artin theorem, one can conclude that $\Bbbk_{par} G$ is a semisimple algebra. Our aim in this section is to determine the simple modules over $\Bbbk_{par} G$ and describe the Grothendieck ring of the category of $\Bbbk_{par}G$-modules. It is worth mentioning that the finite-dimensional simple modules over $\Bbbk_{par}G$ have already been described in \cite{DN}, using the isomorphism
\(
\Bbbk_{par}G \simeq \Bbbk\Gamma(G).
\)
There, the authors describe the simple $\Bbbk\Gamma(G)$-modules in terms of tensor products over $\Bbbk H$, where $H$ ranges over the subgroups of $G$, with the choice of $H$ depending on the simple module under consideration. { In the present section, however, we use the idempotents $\Gamma_X$ associated with each $X\in\mathcal{P}_e(G)$, which provide a decomposition of $\Bbbk_{par}G$ into the corresponding components.
}

First, recall that the set
\[
\mathcal{B}=\{P_X \in A_{par}(\Bbbk G) \mid X \in \mathcal{P}_e(G)\}
\]
forms a linear basis of the algebra $A_{par}(\Bbbk G)$. For each $X \in \mathcal{P}_e(G)$, define the subspace
\[
M_X = \mathrm{span}_{\Bbbk}\{P_{gX} \mid g^{-1} \in X\}.
\]

Let
\[
G_X = \{g \in G \mid gX = X\}
\]
denote the isotropy subgroup of the subset $X$. Choose representatives
$g_1^{-1}, \ldots, g_k^{-1}$ of the right cosets of $G_X$ in $G$, so that
\[
G = \dot{\bigcup}_{i=1}^{k} G_X g_i^{-1}.
\]
Without loss of generality, we may assume that \(X = \dot{\bigcup}_{i=1}^{m} G_X g_i^{-1},
\) for some $m \leq k$. Now, for any $h \in G$, there exist $l \in \{1,\ldots,k\}$ and $s \in G_X$
such that $h = g_l s$. Consequently, \( hX = g_l s X = g_l X.\) Moreover, one can verify that $h^{-1} \in X$ if and only if $l \leq m$. It follows that \( \{P_{g_i X} \mid 1 \leq i \leq m\} \)
forms a basis of the vector space $M_X$. Now consider $g \in G$ and an element $P_{g_i X} \in M_X$, with
$i \in \{1,\ldots,m\}$. Then there exist $j \in \{1,\ldots,k\}$ and
$t \in G_X$ such that $g g_i = g_j t$. Hence,
\[
g g_i X = g_j t X = g_j X,
\]
and notice that the following are equivalent:
\[
 g_j^{-1} \in X, \ \ \ \ j \in \{1, ..., m\}, \ \ \textrm{ and } \ \ g^{-1} \in g_iX.
\]

Finally, consider an irreducible representation $\pi_\alpha :G_X \rightarrow GL( \mathbb{V}^{(\alpha)})$ of the isotropy subgroup of $X$ and define
\[
M_{(X,\alpha)}=M_X \otimes \mathbb{V}^{(\alpha)}.
\]

\begin{proposition}\label{proposition.Mxalpha.module.structure}
The space $M_{(X,\alpha)}$ has the structure of $\Bbbk_{par}G$-module given by
\[
[g] \triangleright \left(P_{g_iX} \otimes v\right)= \left\{ \begin{array}{lc} P_{g_jX} \otimes \pi_\alpha (t)v, & \text{if } g_j^{-1}\in X, \\
0, & \text{otherwise,}\end{array}\right.
\]
in which $gg_i=g_jt$, for $i \in \{1,...,m\}$, $j \in \{1,...,k\}$ and $t\in G_X$.
\end{proposition}

\begin{proof}
    Define the map $\varphi: G\rightarrow \text{End}_\Bbbk (M_X \otimes \mathbb{V}^{(\alpha)})$ given by
    \[
    \varphi (g)(P_{g_iX}\otimes v)=P_{g_jX}\otimes \pi_\alpha (t)v,
    \]
    for $gg_i=g_jt$, with $i \in \{1,...,m\}$, $j\in \{1,...,k\}$ and $t\in G_X$. Recall that that if $g_j^{-1} \notin X$, then $P_{g_jX} = 0$. Now, for $g,h,k\in G$, $l \in \{1, ..., m\}$ and $v\in \mathbb{V}^{(\alpha)}$ we have
    \begin{eqnarray*}
    & \, & \varphi (g)\circ \varphi (h) \circ \varphi (h^{-1})(P_{g_lX}\otimes v) =\\
    & \stackrel{h^{-1}g_l=g_{r}s}{=} & \varphi (g)\circ \varphi (h) (P_{g_rX}\otimes \pi_\alpha (s)v)\; [\![ g_r^{-1}\in X]\!] \\
    & \stackrel{hg_r=g_ls^{-1}}{=}& \varphi (g) (P_{g_lX}\otimes \pi_\alpha (s^{-1})\pi_\alpha (s)v)\; [\![ g_{r}^{-1}\in X]\!] \\
    & \stackrel{gg_{l}=g_ty}{=}& P_{g_tX}\otimes \pi_\alpha (y)v\; [\![ g_r^{-1}\in X]\!] \; [\![ g_t^{-1}\in X]\!].
    \end{eqnarray*}
    On the other hand,
    \begin{eqnarray*}
        & \, & \varphi (gh) \circ \varphi (h^{-1})(P_{g_lX}\otimes v) =\\
        & \stackrel{h^{-1}g_l=g_rs}{=} & \varphi (gh) (P_{g_rX}\otimes \pi_\alpha (s)v)\; [\![ g_r^{-1}\in X]\!] \\
        & \stackrel{(\ast )}{=} & P_{g_tX}\otimes \pi_\alpha (ys^{-1}) \pi_\alpha (s)v\; [\![ g_r^{-1}\in X]\!] \; [\![ g_t^{-1}\in X]\!] \\
        & = & P_{g_tX}\otimes \pi_\alpha (y)v\; [\![ g_r^{-1}\in X]\!] \; [\![ g_t^{-1}\in X]\!],
    \end{eqnarray*}
    in which $(\ast )$ indicates the equalities $ghg_r=gg_ls^{-1}=g_tys^{-1}$. Therefore
    \[
    \varphi (g) \circ \varphi (h) \circ \varphi (h^{-1})=\varphi (gh) \circ \varphi (h^{-1}).
    \]
    Analogously, one can verify that
    \[
    \varphi (g^{-1})\circ \varphi (g) \circ \varphi (h) =\varphi (g^{-1})\circ \varphi (gh).
    \]
    Therefore, $\varphi: G\rightarrow \text{End}_\Bbbk (M_X \otimes \mathbb{V}^{(\alpha)})$ is a partial representation of $G$, which factorizes through an algebra morphism $\overline{\varphi}:\Bbbk_{par} G \rightarrow \text{End}_\Bbbk (M_X \otimes \mathbb{V}^{(\alpha)})$, making $M_{(X,\alpha)}$ into a $\Bbbk_{par}G$-module.
\end{proof}

\begin{remark}
    One may choose another set of representatives $h_1^{-1},\ldots,h_k^{-1}$ for the right cosets of $G_X$ in $G$. Without loss of generality, we may assume that
\[
G_X h_j^{-1} = G_X g_j^{-1}, \qquad \text{for all } 1 \leq j \leq k.
\]
In this case, the set $\{P_{h_iX} \mid 1 \leq i \leq m\}$ also forms a basis of $M_X$. Consequently, the module $M_{(X,\alpha)}$ admits the following $\Bbbk_{par}G$-module structure:
\[
[g]\blacktriangleright (P_{h_iX}\otimes v) =
\begin{cases}
P_{h_jX}\otimes \pi_\alpha(t)v, & \text{if } h_j^{-1}\in X,\\
0, & \text{otherwise,}
\end{cases}
\]
where $gh_i = h_j t$, with $i \in \{1,\ldots,m\}$, $j \in \{1,\ldots,k\}$, and $t \in G_X$.

Observe that, since we assume $G_X h_j^{-1} = G_X g_j^{-1}$ for all $1 \leq j \leq k$, it follows that $h_j^{-1}g_j \in G_X$. Hence we may consider the linear map
\[
f:M_{(X,\alpha)} \longrightarrow M_{(X,\alpha)},
\]
defined by
\[
f(P_{g_iX}\otimes v) = P_{h_iX}\otimes \pi_\alpha(h_i^{-1}g_i)v,
\ \ \text{for } i \in \{1,\ldots,m\}.
\]

One can verify that $f$ is an isomorphism of $\kpar G$-modules, where the $\kpar G$-module structure on the domain is the one given in Proposition~\ref{proposition.Mxalpha.module.structure}. More precisely, for all $g \in G$ one has
\[
f\bigl([g]\triangleright (P_{g_iX}\otimes v)\bigr)
=
[g]\blacktriangleright f(P_{g_iX}\otimes v).
\]
\end{remark}

\begin{theorem}\label{simplemodules}
    For each subset $X\in \mathcal{P}_e(G)$ and each irreducible representation $\pi_\alpha :G_X \rightarrow GL( \mathbb{V}^{(\alpha)})$, the $\Bbbk_{par}G$-modules $M_{(X,\alpha)}$ are simple. Moreover, any simple $\Bbbk_{par}G$-module is of this form.
\end{theorem}

\begin{proof}
    { Let $g_1,\ldots,g_n\in G$ be such that
\(
X=\dot{\bigcup}_{j=1}^n G_Xg_j^{-1},
\)
and consider $\{P_{g_iX}\mid 1\leq i\leq n\}$ as a basis of $M_X$. Let $P_{g_iX}\otimes v\in M_X\otimes V^{(\alpha)}$ and $g\in G$. Then
\begin{eqnarray*}
\varepsilon_g \triangleright (P_{g_iX}\otimes v) & = & [g][g^{-1}]\triangleright (P_{g_iX}\otimes v) \\
& \stackrel{g^{-1}g_i=g_jt}{=} & [g]\triangleright (P_{g_jX}\otimes \pi_\alpha (t)v)\; [\![g\in g_iX]\!] \\
& \stackrel{gg_j=g_it^{-1}}{=} & (P_{g_iX}\otimes \pi_\alpha (t^{-1})\pi_\alpha (t)v)[\![g^{-1}\in g_jX]\!][\![g\in g_iX]\!] \\
& = & (P_{g_iX}\otimes v)[\![g^{-1}\in g_jX]\!][\![g\in g_iX]\!].
\end{eqnarray*}
Since $g=g_it^{-1}g_j^{-1}$, with $g_i^{-1},g_j^{-1}\in X$, we have that the boolean velues of $[\![g\in g_iX]\!]$ coincide with 
$[\![g^{-1}\in g_jX]\!]$. Consequently, \[
\varepsilon_g\triangleright(P_{g_iX}\otimes v) = [\![g\in g_iX]\!](P_{g_iX}\otimes v).\]
It follows that
\begin{equation}\label{equation.PX.partial.action}
P_X\triangleright(P_Y\otimes v)
=
[\![X=Y]\!](P_Y\otimes v).
\end{equation}
Moreover, for any $a\in\Bbbk G_X$ and $i,j\in{1,\ldots,n}$, we have
\begin{equation}\label{equation.g_iag_j^{-1}.partial.action}
[g_iag_j^{-1}]
\triangleright(P_{g_jX}\otimes v)
=
P_{g_iX}\otimes\pi_\alpha(a)v.
\end{equation}
Now let $0\neq\omega\in M_{(X,\alpha)}$. Then there exist $v_1,\ldots,v_n\in V^{(\alpha)}$ such that
\[
\omega=\sum_{i=1}^nP_{g_iX}\otimes v_i.
\]
Since $\omega\neq0$, there exists $i_0\in \{1,\ldots,n\}$ such that $v_{i_0}\neq0$. Let $j\in\{1,\ldots,n\}$ and $v\in V^{(\alpha)}$. Since $V^{(\alpha)}$ is a simple $\Bbbk G_X$-module, the submodule generated by $v_{i_0}$ is all of $V^{(\alpha)}$. Hence, there exists $a\in\Bbbk G_X$ such that
\(
\pi_\alpha(a)v_{i_0}=v.
\)
Using the Equations \eqref{equation.PX.partial.action} and \eqref{equation.g_iag_j^{-1}.partial.action}, we obtain
$$\begin{array}{ccl}
[g_ja g_{i_0}^{-1}]P_{g_{i_0}X}\triangleright\omega
& = & \displaystyle
\sum_{i=1}^n
[g_ja g_{i_0}^{-1}]P_{g_{i_0}X}
\triangleright(P_{g_iX}\otimes v_i) \\
& = &
P_{g_jX}\otimes\pi_\alpha(a)v_{i_0} \ = \
P_{g_jX}\otimes v.
\end{array}$$
Therefore, every element of the form $P_{g_jX}\otimes v$ belongs to the $\Bbbk_{par}G$-submodule generated by $\omega$. Since these elements span $M_{(X,\alpha)}$, we conclude that
\(
\Bbbk_{par}G\triangleright\omega=M_{(X,\alpha)}.
\)
Thus, $M_{(X,\alpha)}$ is a simple $\Bbbk_{par}G$-module.
    }

    Let us prove now that any simple module of $\kpar G$ is of the form $M_{(X,\alpha)}$. Recall that $T \subseteq \mathcal{P}_e(G)$ is a fundamental domain for the equivalence relation between subsets performed by translation. That is, if $Y,Z \in T$, with $Y\neq Z$, then $Y\nsim Z$, and also, for any $X\in \mathcal{P}_e(G)$, there exists $Y \in T$ with $X\sim Y$. Then, by Theorem \ref{theorem.kparGGamma.iso.MnKG} 
    \[
    \kpar G \simeq \bigoplus_{X\in T} \kpar G\, \Gamma_X \simeq \bigoplus_{X\in T} M_{n_X}(\Bbbk G_X),
    \]
    where $n_X = \frac{|X|}{|G_X|}$. So any simple $\kpar G$-module is a simple module of one of the components $M_{n}(\Bbbk G_X)$, for some $X$, where $n=\frac{|X|}{|G_X|}$. Since $M_{n}(\Bbbk G_X)$ and $\Bbbk G_X$ are Morita equivalent, if follows that any simple $M_{n}(\Bbbk G_X)$-module is of the form $(\mathbb{V}^{(\alpha)})^{n}$, where we have a irreducible representation $\pi_{\alpha}:G_X \to GL(\mathbb{V}^{(\alpha)})$ and the $M_{n}(\Bbbk G_X)$ action on $(\mathbb{V}^{(\alpha)})^{n}$ is by left matrix multiplication. Now to write down the $\kpar G$ action on $(\mathbb{V}^{(\alpha)})^{n}$ in terms of classes of $[g]$, first we consider $g_1, ..., g_{n} \in G$ such that $\displaystyle \Gamma_X = \sum_{i=1}^{n}P_{g_iX}$, and recall that 
    \[
    \begin{array}{cccl}
       \varphi_X: & \kpar G & \longrightarrow & M_{n}(\Bbbk G_X) \\
         & [g] & \longmapsto & \displaystyle \sum_{i,j=1}^{n}\left(\bool{g_i^{-1}gg_j \in G_X}g_i^{-1}gg_j\right)E_{ij}
    \end{array}
    \]
    is an algebra map that when restricted to $\kpar G\, \Gamma_X $ becomes an isomorphism of algebras (see Theorem \ref{theorem.kparGGamma.iso.MnKG}). So the $\kpar G$-action on $(\mathbb{V}^{(\alpha)})^{n}$ can be written as follows,
    \[
    [g] \triangleright \left[\begin{array}{c}
         v_1   \\
         \vdots \\
         v_{n}
    \end{array}\right] = \varphi_X([g])\left[\begin{array}{c}
         v_1   \\
         \vdots \\
         v_{n}
    \end{array}\right] = \left[\begin{array}{c}
         \sum_{j=1}^{n}{\small \bool{g_1^{-1}gg_j\in G_X}}\pi_{\alpha}(g_1^{-1}gg_j)v_j   \\
         \vdots \\
         \sum_{j=1}^{n}{\small \bool{g_{n}^{-1}gg_j\in G_X}}\pi_{\alpha}(g_{n}^{-1}gg_{j})v_{j}
    \end{array}\right].
    \]
    Now to realize that this $\kpar G$-module is the module $M_{(X,\alpha)} = M_X\otimes \mathbb{V}^{(\alpha)}$, defined in Proposition \ref{proposition.Mxalpha.module.structure}, consider the basis $\{P_{g_1X}, ..., P_{g_{n}X}\}$ of $M_X$, and in terms of that basis, we can rewrite the $\kpar G$ action this way:
    \[
    [g]\triangleright (P_{g_jX}\otimes v) =  
    \left\{\begin{array}{ll} 
    P_{g_iX} \otimes \pi_\alpha (t)v, & \text{if } gg_j = g_it, \textrm{ for some } t \in G_X, \\
0, & \text{otherwise},\end{array}\right.
    \]
    for any $i, j \in \{1,...,n\}$. Notice that if we take $i,j,k \in \{1,..., n\}$, such that $g_{i}^{-1}gg_j, g_k^{-1}gg_j \in G_X$, then $g_{j}^{-1}g^{-1} \in G_{X}g^{-1}_k\cap G_{X}g_i^{-1}$, hence $k=i$. So we can write the $\kpar G$-action as follows:
    \[
    [g]\triangleright (P_{g_jX}\otimes v) = \sum_{i=1}^n\bool{g_i^{-1}gg_j \in G_X}P_{g_iX}\otimes \pi_{\alpha}(g_i^{-1}gg_j)v.
    \]
    Now consider the linear bijection
    
    \[
    \begin{array}{cccc}
        f: & (\mathbb{V}^{(\alpha)})^{n} & \longrightarrow & M_{(X,\alpha)} \\
         & \left[\begin{array}{c}
         v_1   \\
         \vdots \\
         v_{n}
    \end{array}\right] & \longmapsto & \displaystyle \sum_{j=1}^{n} P_{g_jX}\otimes v_j 
    \end{array}
    \]
    finally, it is not hard to see that $f: (\mathbb{V}^{(\alpha)})^{n} \to M_{(X,\alpha)}$ is $\kpar G$-linear, so $(\mathbb{V}^{(\alpha)})^n \simeq M_{(X,\alpha)}$ as $\kpar G$-module.
\end{proof}

\begin{remark}
    There is another equivalent characterization of the simple $\Bbbk_{par}G$-modules given in Theorem 2.3 of \cite{DN}, but, for our purposes the characterization of simple modules in the Theorem \ref{simplemodules} is a more useful and interesting way to study partial representations of $G$.  
\end{remark}

In order to describe the tensor product between simple objects, first recall Theorem \ref{Teorema-Schauenburg}, which implies that, for a weak Hopf algebra $H$ and for every left $H$-modules $M$ and $N$, we have the isomorphism of left $H$-modules
\[
M\otimes_{H_t} N \cong \Delta (1)(M\otimes N)=M\boxtimes N.
\]
In our case, using the isomorphism $\lambda :\Bbbk_{par} G \rightarrow \Bbbk \Gamma (G)$, we have the balanced tensor product between the left $\Bbbk_{par}G$-modules $M$ and $N$ is given by \(M\otimes_{A_{par}(\Bbbk G)} N \cong \Delta(1_{\Bbbk_{par} G})(M\otimes N)\). Since the isomorphism $\lambda:\kpar G \to \Bbbk\Gamma(G)$ is, in particular, an isomorphism of weak Hopf algebras, it follows that 
\[
\Delta(1_{\kpar G}) = \sum_{X \in \mathcal{P}_e(G)}P_X\otimes P_X,
\]
where we recall that $\lambda^{-1}((e,X)) = P_X$ and $\displaystyle \Delta(1_{\Bbbk \Gamma(G)}) = \sum_{x \in \Gamma(G)^{(0)}}x\otimes x$ (see Example \ref{exemplo-kgama-bialgebrafraca}). Hence,
\begin{equation}\label{equation.M.tensor.Apar.N.simeq.UmMN}
    M\otimes_{A_{par}(\Bbbk G)} N \cong \sum_{X\in \mathcal{P}_e(G)} (P_X \triangleright M) \otimes (P_X \triangleright N).
\end{equation}

{
The following result will be useful for characterizing the tensor \ product of $\kpar G$-modules.
\begin{lemma}\label{lemma.MXV.tensor.MXW.simeq.MXVtensorW}
    Let $X \in \mathcal{P}_e(G)$, and let $\mathbb{V}$ and $\mathbb{W}$ be two $\Bbbk G_X$-modules. 
    \[ \bigl(M_X\otimes \mathbb{V}\bigr)\otimes_{A_{par}(\Bbbk H)} \bigl(M_X\otimes \mathbb{W}\bigr) \simeq M_X\otimes\bigl(\mathbb{V}\otimes \mathbb{W}\bigr).
    \]
\end{lemma}
\begin{proof}
    Expressing the isomorphism in Equation \eqref{equation.M.tensor.Apar.N.simeq.UmMN} explicitly, we have 
    \[
   \begin{array}{ccc}
       M\otimes_{A_{par}(\Bbbk G)} N & \stackrel{\phi}{\longrightarrow} & \displaystyle \sum_{X\in \mathcal{P}_e(G)} (P_X \triangleright M) \otimes (P_X \triangleright N) \\
       m\otimes_{A_{par}(\Bbbk G)}n & \longmapsto & \displaystyle \sum_{X\in \mathcal{P}_e(G)} (P_X \triangleright m) \otimes (P_X \triangleright n)
   \end{array}
    \]
    for any $\kpar G$-modules $M$ and $N$. Now, for all $P_{gX}, P_{hX} \in M_X$, $v \in \mathbb{V}$ and $w \in \mathbb{W}$, applying the bijection $\phi$ to the element $$(P_{gX}\otimes v)\otimes_{A_{par}(\Bbbk G)}(P_{hX}\otimes w) \in \bigl(M_X\otimes \mathbb{V}\bigr)\otimes_{A_{par}(\Bbbk G)} \bigl(M_X\otimes \mathbb{W}\bigr), $$
    we have 
\begin{eqnarray*}
& \; \vspace{-3.0cm} \displaystyle \sum_{Z\in \mathcal{P}_e(G)} P_Z \triangleright \left( P_{gX}\otimes v \right) \otimes P_Z \triangleright \left( P_{hX}\otimes w \right) =  \\
& \stackrel{\eqref{equation.PX.partial.action}}{=} (P_{gX}\otimes v) \otimes (P_{hX}\otimes w)\bool{gX = hX}.
\end{eqnarray*}
Then the element $(P_{gX}\otimes v)\otimes_{A_{par}(\Bbbk G)}(P_{hX}\otimes w)$ vanishes if $gX \neq hX$.

Let $g_1, ..., g_n \in G$ such that \( X=\dot{\bigcup}_{j=1}^n G_Xg_j^{-1}, \) and consider $\{P_{g_iX} \mid 1 \leq i \leq n\}$ as a basis of $M_X$. Then any vector $\nu \in \bigl(M_X\otimes \mathbb{V}\bigr)\otimes_{A_{par}(\Bbbk H)} \bigl(M_X\otimes \mathbb{W}\bigr)$, there exist $v_1,...,v_n \in \mathbb{V}$ and $w_1,...,w_n \in \mathbb{W}$ such that 
\[
\nu = \sum_{i,j=1}^n(P_{g_iX}\otimes v_i)\otimes_{A_{par}(\Bbbk G)}(P_{g_jX}\otimes w_j) = \sum_{i=1}^n(P_{g_iX}\otimes v_i)\otimes_{A_{par}(\Bbbk G)}(P_{g_iX}\otimes w_i)
\]
then consider the following surjective map
\[
\begin{array}{cccl}
    f: & M_X\otimes\bigl(\mathbb{V}\otimes \mathbb{W}\bigr) & \longrightarrow & \bigl(M_X\otimes \mathbb{V}\bigr)\otimes_{A_{par}(\Bbbk G)} \bigl(M_X\otimes \mathbb{W}\bigr) \\
     & P_{g_iX}\otimes (v \otimes w) & \longmapsto & \bigl(P_{g_iX}\otimes v\bigr)\otimes_{A_{par}(\Bbbk G)}\bigl(P_{g_iX}\otimes w\bigr)
\end{array}
\]
Now, we are going to see that $f$ is $\kpar G$-linear. First, recall that the monoidal structure of the category of modules over a Hopf Algebroid is given by the comultiplication, which in our case, for any $g \in G$, 
$$\Delta_r([g]) = \Delta_l([g]) = [g]\otimes_{A_{par}(\Bbbk G)}[g].$$
Take $ g \in G $, $v\in \mathbb{V}$, and $w \in \mathbb{W}$ and notice that for any $i \in \{1,\dots ,n\}$, 
\begin{eqnarray*}
    f\bigl([g]\triangleright (P_{g_iX}\otimes (v\otimes w))\bigr) & \stackrel{gg_i=g_jt}{=} & f\bigl(P_{g_jX}\otimes t\triangleright (v\otimes w)\bigr) \\
     & = & f\bigl(P_{g_jX}\otimes (t\triangleright v)\otimes (t\triangleright w)\bigr) \\
     & = & (P_{g_jX}\otimes t\triangleright v) \otimes_{A_{par}(\Bbbk H)} (P_{g_jX}\otimes t\triangleright w) \\
     & = & \bigl([g]\triangleright(P_{g_i}\otimes v)\bigr)\otimes_{A_{par}(\Bbbk H)}\bigl([g]\triangleright(P_{g_i}\otimes w)\bigr) \\
     & = & [g]\triangleright f\bigl((P_{g_iX}\otimes (v\otimes w))\bigr),
\end{eqnarray*}
take $\{v_i \mid i \in I\}$ and $\{w_j \mid i \in J\}$ bases of $\mathbb{V}$ and $\mathbb{W}$, respectively. Then, notice that $\{P_{g_kX}\otimes v_i\otimes w_j \mid k \in \{1,...,n\}, i \in I, j \in J\}$ is a basis of $M_X\otimes \mathbb{V}\otimes\mathbb{W}$. Now to see that $\{(P_{g_kX}\otimes v_i)\otimes_{A_{par}(\Bbbk G)}(P_{g_kX}\otimes w_j) \mid k \in \{1,...,n\}, i \in I, j \in J\}$ is a basis of $\bigl(M_X\otimes \mathbb{V}\bigr)\otimes_{A_{par}(\Bbbk G)} \bigl(M_X\otimes \mathbb{W}\bigr)$, observe that the image of that set by the bijection $\phi$ is the set
\[
\{(P_{g_kX}\otimes v_i)\otimes (P_{g_kX}\otimes w_j) \mid k \in \{1,...,n\}, i \in I, j \in J\},
\]
which is a linear independent set. Finally, since $f:M_X\otimes\bigl(\mathbb{V}\otimes \mathbb{W}\bigr) \to \bigl(M_X\otimes \mathbb{V}\bigr)\otimes_{A_{par}(\Bbbk G)} \bigl(M_X\otimes \mathbb{W}\bigr)$ maps a basis to a basis, it is an isomorphism of $\kpar G$-modules.
\end{proof}
}

We have that for simple modules $M_X \otimes \mathbb{V}^{(\alpha)}$ and $M_Y \otimes \mathbb{V}^{(\beta)}$,
\begin{eqnarray*}
& \; & \left( M_X \otimes \mathbb{V}^{(\alpha)} \right) \otimes_{A_{par}(\Bbbk G)} \left( M_Y\otimes \mathbb{V}^{(\beta)} \right)  \\
& \stackrel{\eqref{equation.M.tensor.Apar.N.simeq.UmMN}}{\cong} & \sum_{Z\in \mathcal{P}_e(G)} P_Z \triangleright \left( M_X \otimes \mathbb{V}^{(\alpha)} \right) \otimes P_Z \triangleright \left( M_Y\otimes \mathbb{V}^{(\beta)} \right) \\
& \stackrel{\eqref{equation.PX.partial.action}}{=} & \sum_{Z\in \mathcal{P}_e(G)} P_Z \triangleright \left( M_X \otimes \mathbb{V}^{(\alpha)} \right) \otimes P_Z \triangleright \left( M_Y\otimes \mathbb{V}^{(\beta)} \right) [\! [ X\sim Y ]\! ] \\
& \stackrel{\eqref{equation.M.tensor.Apar.N.simeq.UmMN}}{\cong} & \left( \left( M_X \otimes \mathbb{V}^{(\alpha)} \right) \otimes_{A_{par}(\Bbbk G)} \left( M_X\otimes \mathbb{V}^{(\beta)} \right)\right)[\! [ X\sim Y ]\! ] \\
& \stackrel{\ref{lemma.MXV.tensor.MXW.simeq.MXVtensorW}}{\cong} & \left(M_X \otimes \left( \mathbb{V}^{(\alpha)} \otimes \mathbb{V}^{(\beta)} \right)\right) [\! [ X\sim Y ]\! ] \\
& \cong & \left( M_X \otimes \bigoplus_{\gamma \in \widehat{G_X}} n_\gamma^{\alpha \, \beta } \mathbb{V}^{(\gamma)}\right) [\! [ X\sim Y ]\! ] ,
\end{eqnarray*}
in which the coefficients $n_\gamma^{\alpha \, \beta }$ are the respective multiplicities of $\mathbb{V}^{(\gamma)}$ and the direct sum goes through the set of equivalent irreducible representations of the isotropy subgroup $G_X$.

The monoidal unit in the monoidal category ${}_{\Bbbk_{par}G}\text{Mod}$ is the algebra $A_{par}(\Bbbk G)$. More specifically, $A_{par}(\Bbbk G)$ can be viewed as the direct sum
\begin{equation}\label{equation.Apar.SomaDireta}
    A_{par}(\Bbbk G)=\bigoplus_{X\in T } \left(M_X \otimes \Bbbk_{\epsilon} \right) ,
\end{equation}
where $T\subseteq \mathcal{P}_e (G)$ is a fundamental domain of $\mathcal{P}_e (G)$ relative to the equivalence relation of subsets by translation. The one dimensional space $\Bbbk_{\epsilon}$ denotes the space carrying the trivial representation of the isotropy subgroup $G_X$: 
\[
\begin{array}{rccl} \epsilon : & G_X & \rightarrow & \Bbbk^{\times } \\ \, & g & \mapsto & 1 \end{array} .
\]
Therefore, the category ${}_{\Bbbk_{par} G}\text{Mod}$ is a multifusion category, because the unit object in the category is not simple.

{
The following result summarizes what we have been doing. 

\begin{proposition}
    Let $G$ be a finite group, $\Bbbk$ an algebraically closed field with $\textrm{char}(\Bbbk) \nmid |G|$, and $T \subseteq P_{e}(G)$  a fundamental domain of $\mathcal{P}_e (G)$ relative to the equivalence relation of subsets by translation. Then 
    \[
    \mathcal{K}({ }_{\kpar G}\text{Mod}) \simeq \bigoplus_{X \in T}\mathcal{K}({ }_{\Bbbk G_X}\text{Mod})
    \]
    where $\mathcal{K}(\mathcal{C})$ is the Grothendieck ring of the monoidal category $\mathcal{C}$.
\end{proposition}
\begin{proof}
    To prove that fact, we only have to summarize what we already know. First, since $\kpar G$ is a semisimple algebra, the Grothendieck ring of $\kpar G$-modules is generated by the isomorphism classes of the simple $\kpar G$-modules. By Theorem \ref{simplemodules}, any simple $\kpar G$-module is of the form $M_{(X,\alpha)}$, for some $X \in \mathcal{P}_e(G)$ and some irreducible representation $\pi_{\alpha}:G_X \to GL(\mathbb{V}^{\alpha})$. By what we have done so far, it follows that 
    \[
    [M_{(X,\alpha)}][M_{(Y,\beta)}] = [M_{(X,\alpha)}\otimes_{A_{par}(\Bbbk G)}M_{(Y,\beta)}]
    \]
    vanishes if $X \not\sim Y$ and, for a fixed $X$,
    \begin{eqnarray*}
        M_{(X,\alpha)}\otimes_{A_{par}(\Bbbk G)} M_{(X,\beta)} & = & \bigl(M_X\otimes \mathbb{V}^{(\alpha)}\bigr)\otimes_{A_{par}(\Bbbk G)}\bigl(M_X\otimes \mathbb{V}^{(\beta)}\bigr) \\
        & \stackrel{\ref{lemma.MXV.tensor.MXW.simeq.MXVtensorW}}{\simeq} & M_X \otimes\bigl(\mathbb{V}^{(\alpha)}\otimes \mathbb{W}^{(\beta)}\bigr).
    \end{eqnarray*} 
    Now, for the unit of $\mathcal{K}({ }_{\kpar G}\text{Mod})$, notice that 
    \[
    1_{\mathcal{K}({ }_{\kpar G}\text{Mod})} = [A_{par}(\Bbbk G)] \stackrel{\eqref{equation.Apar.SomaDireta}}{=} \bigoplus_{X\in T}[\left(M_X \otimes \Bbbk_{\epsilon} \right)]  ,
    \]
    and $\Bbbk_{\epsilon}$ is the object unit on the monoidal category ${ }_{\Bbbk G_X}\text{Mod}$, which conclude the proof.
\end{proof}
}

\begin{example}
    Consider the cyclic group $C_3=\langle g \; | \; g^3 =e \rangle =\{ e, g ,g^2 \}$. The set $\mathcal{P}_e (C_3)$ is given by
    \[
    \mathcal{P}_e (C_3)=\left\{ \{ e\} ,X_1=\{e,g\}, X_2=\{ e,g^2\} , C_3\right\} .
    \]
    It is easy to see that $G_{\{ e\}} =G_{X_1}=G_{X_2}=\{ e\}$ and $G_{C_3}=C_3$ and that $X_2 =g^2X_1$ and $X_1=gX_2$. For the trivial isotropy subgroup we have only one irreducible representation $\epsilon :\{ e \} \rightarrow \Bbbk^{\times}$ sending $e$ to $1$. For the isotropy subgroup $C_3$, we have three one dimensional irreducible representations:
    $$ \xymatrix@C=1cm@R=0.1cm{
					\rho_0:G\ar[r] &  \Bbbk ^{\times } \ ,   & \rho_1:G\ar[r] &  \Bbbk ^{\times } \ , & \rho_2:G\ar[r] &  \Bbbk ^{\times } , \\
					\ \ \ e \ar@{|->}[r] & 1 & \ \ \ e \ar@{|->}[r] & 1  & \ \ \ e \ar@{|->}[r] & 1 \\
					\ \ \  g \ar@{|->}[r] & 1 & \ \ \ g \ar@{|->}[r] & \omega  & \ \ \ g \ar@{|->}[r] & \omega^2 \\
					\ \ \  g^2 \ar@{|->}[r] & 1 & \ \ \ g^2 \ar@{|->}[r] & \omega^2  & \ \ \ g^2 \ar@{|->}[r] & \omega 
				}$$
    where $\omega \in \Bbbk$ is a primitive cubic root of $1$.          
    
    Therefore, we have the following simple $\Bbbk_{par}C_3$-modules:
    \begin{enumerate}
        \item The one dimensional trivial $\Bbbk_{par}C_3$ module 
        \[
        M_{(\{e\},\epsilon )} =M_{\{e\}}\otimes \Bbbk_{\epsilon} =\Bbbk P_{\{e\}} .
        \]
        In this representation
        \[
        [e]\triangleright P_{\{e\}}=P_{\{e\}}, \quad [g]\triangleright P_{\{ e\}} =[g^2]\triangleright P_{\{ e\}} =0.
        \]
        \item The two dimensional $\Bbbk_{par}C_3$ module
        \[
        M_{(X_1 ,\epsilon)}= M_{X_1}\otimes \Bbbk_{\epsilon} =\text{span}_{\Bbbk} \{ P_{X_1} ,P_{X_2}\} .
        \]
        The matricial representation of the generators of $\Bbbk_{par}C_3$ are given by
        \[
        [e]=\left( \begin{array}{cc} 1 & 0 \\ 0 & 1 \end{array} \right) , \quad [g]=\left( \begin{array}{cc} 0 & 1 \\ 0 & 0 \end{array} \right) , \quad [g^2]=\left( \begin{array}{cc} 0 & 0 \\ 1 & 0 \end{array} \right) .
        \]
        \item Three one dimensional $\Bbbk_{par}C_3$ modules:
        \begin{enumerate}
            \item $M_{(C_3,0)}= M_{C_3} \otimes \Bbbk_0=\Bbbk P_{C_3}$;
            \item $M_{(C_3,1)}= M_{C_3} \otimes \Bbbk_1=\Bbbk P_{C_3}$;
            \item $M_{(C_3,2)}=M_{C_3} \otimes \Bbbk_2=\Bbbk P_{C_3}$.
        \end{enumerate}
        These are the usual one dimensional simple $\Bbbk C_3$-modules corresponding to the irreducible representations $\rho_0$, $\rho_1$ and $\rho_2$, respectively.
    \end{enumerate}

    The tensor product between simple $\Bbbk_{par}C_3$ modules can be expressed by the following table:
    $$ \begin{tabular}{ | c | c |  c | c | c | c |}
				\hline
				$\boxtimes$  &  $M_{(\{ e \}, \epsilon) }$ &  $M_{(X_{1}, \epsilon)}$ & $M_{(C_3,0) }$ & $M_{(C_3,1) }$  & $M_{(C_3,2) }$  \\ \hline 
				$M_{(\{ e \}, \epsilon) }$  & $M_{(\{ e \}, \epsilon) }$ & $0$ & $0$ & $0$  & $0$  \\ \hline 
				$M_{(X_{1},\epsilon)}$  & $0$ & $M_{(X_{1},\epsilon)}$ & $0$ & $0$  & $0$  \\ \hline \cline{4 - 6}
				$M_{(C_3,0) }$  & $0$ & $0$ & $M_{(C_3,0) }$ & $M_{(C_3,1) }$ & $M_{(C_3,2) }$  \\ \hline \cline{4 - 6}
				$M_{(C_3,1) }$  & $0$ & $0$ & $M_{(C_3,1) }$ & $M_{(C_3,2) }$ & $M_{(C_3,0) }$  \\ \hline \cline{4 - 6} 
				$M_{(C_3,2) }$  & $0$ & $0$ & $M_{(C_3,2) }$ & $M_{(C_3,0) }$ & $M_{(C_3,1) }$  \\ \hline \cline{4 - 6}
			\end{tabular}$$
\end{example}

\section{Christmas Tree's and Matryoshka's Theorems}

For each subgroup $H\leq G$ and for each representation $\pi: H\rightarrow GL(\mathbb{V})$, the $\Bbbk_{par}G$-module $M_H\otimes \mathbb{V}\cong \mathbb{V}$ can be viewed as a usual $\Bbbk H$-module. Indeed, given $g\in G$, and $v\in \mathbb{V}$, we have  
\[
[g]\triangleright (P_H\otimes v)=(P_H \otimes \pi (g)v)\; [\![g\in H]\!] .
\]
Moreover, for a morphism of $\Bbbk H$-modules $f:\mathbb{V} \rightarrow \mathbb{W}$, the associated linear map of $\Bbbk_{par}G$-modules
\[
M_H \otimes f: M_H \otimes \mathbb{V}\rightarrow M_H\otimes \mathbb{W}
\]
is a morphism of $\Bbbk_{par}G$-modules. Indeed, for $g\in G$ and $v\in \mathbb{V}$
\begin{eqnarray*}
    & \, & (M_H\otimes f)([g]\triangleright (P_H \otimes v))  =  (M_H\otimes f) (P_H \otimes \pi^{\mathbb{V}}(g)v)\; [\![ g\in H]\!] \\
    & = & (P_H \otimes f(\pi^{\mathbb{V}}(g)v)) \; [\![ g\in H]\!]  =  (P_H \otimes \pi^{\mathbb{W}}(g)f(v)) \; [\![ g\in H]\!] \\
    & = &[g]\triangleright (P_H \otimes f(v)) =  [g]\triangleright \left((M_H\otimes f)(P_H \otimes v)\right).
\end{eqnarray*}
Then, the category of $\Bbbk H$-modules can be viewed as a subcategory of the category of $\Bbbk_{par}G$-modules, as stated in the following theorem.

\begin{theorem}[Christmas Tree's Theorem]\label{natal}
Let $G$ be a finite group and $H \leq G$ a subgroup. The construction that associates to each $\Bbbk H$-module $\mathbb{V}$ the $\kpar G$-module $M_H\otimes_{\Bbbk H}\mathbb{V}$ induces a functor \[\mathcal{F} : { }_{\Bbbk H}\textnormal{Mod} \longrightarrow { }_{\Bbbk_{par} G}\textnormal{Mod}.\]
This functor is additive, monoidal, essentially injective and faithful. 
\end{theorem}
\begin{proof}
We define the functor $\mathcal{F}: { }_{\Bbbk H}\text{Mod} \longrightarrow { }_{\Bbbk_{par} G}\text{Mod}$ in the following way: for every left $\Bbbk H$-module $\mathbb{V}$ associate the left $\Bbbk _{par}G$-module 
\[ \mathcal{F}(\mathbb{V})= M_H \otimes \mathbb{V},\]
where its $\Bbbk _{par}G$-module structure is given by
   \[ [g]\triangleright \Big( P_H \otimes v \Big) = 
         P_H \otimes (g\triangleright v)\; [\![g\in H]\!]  
     \]
    for all $[g]\in \Bbbk _{par}G$. Also, for every morphism of left $\Bbbk H$-modules $f: \mathbb{V} \longrightarrow \mathbb{W}$, then 
    \[\mathcal{F}(f): M_H \otimes \mathbb{V} \longrightarrow M_H \otimes \mathbb{W}\]
    is defined by 
    \[ \mathcal{F}(f)(P_H \otimes v)=P_H \otimes f(v)\]
    which was proved to be a morphism of left $\Bbbk_{par} G$-modules. 
    
The functor $\mathcal{F}$ is additive. Indeed, as $\text{char}\Bbbk =0$ the category ${ }_{\Bbbk H}\text{Mod}$ is semisimple. Consider  $\mathbb{V} =\bigoplus_{i\in I} \mathbb{V}_i \in { }_{\Bbbk H}\text{Mod}$ with $\mathbb{V}_i \in { }_{\Bbbk H}\text{Mod}$ for all $i\in I$. Then 
\[
    \mathcal{F}\big( \bigoplus_{i\in I} \mathbb{V}_i\big) = M_H \otimes \big( \bigoplus_{i\in I} \mathbb{V}_i \big)\stackrel{(\ast)}{\cong} \bigoplus_{i\in I} \big( M_H \otimes  \mathbb{V}_i \big)\\
    = \bigoplus_{i \in I} \big( \mathcal{F}(\mathbb{V}_i ) \big).
\]
By a straightforward calculation, it is easy to verify that the isomorphism $(\ast )$ is an isomorphism of $\Bbbk_{par}G$-modules. 
One can prove also that $\mathcal{F}$ is a monoidal functor as well. First, take $\mathbb{V}$ and $\mathbb{W}$ two $\Bbbk H$-modules, and consider the $\kpar G$-modules
\[
\mathcal{F}(\mathbb{V})\otimes_{A_{par}(\Bbbk G)}\mathcal{F}(\mathbb{W}) = (M_H\otimes \mathbb{V})\otimes_{A_{par}(\Bbbk G)}(M_H\otimes \mathbb{W}),
\]
and 
\[
\mathcal{F} \left( \mathbb{V} \otimes \mathbb{W} \right) =M_H \otimes \left( \mathbb{V} \otimes \mathbb{W} \right) .
\]
By Lemma \ref{lemma.MXV.tensor.MXW.simeq.MXVtensorW}, we have the following isomorphism of $\kpar G$-module
\[
\begin{array}{rcl}
    \varphi_{\mathbb{V},\mathbb{W}}: (M_H\otimes \mathbb{V})\otimes_{A_{par}(\Bbbk G)}(M_H\otimes \mathbb{W}) & \longrightarrow & M_H\otimes (\mathbb{V}\otimes\mathbb{W})\\
     (P_H\otimes v)\otimes_{A_{par}(\Bbbk G)}(P_H\otimes w) & \longmapsto & P_H\otimes (v\otimes w),
\end{array}
\]
Even more \[\varphi_{\underline{\quad} \, ,\, \underline{\quad}}:\mathcal{F}(\underline{\quad})\otimes_{A_{par}(\Bbbk G)}\mathcal{F}(\underline{\quad}) \Rightarrow \mathcal{F(\underline{\quad}\otimes \underline{\quad})} \] is natural. In fact, take $f:\mathbb{V} \to \mathbb{V}'$ and $g:\mathbb{W} \to \mathbb{W}'$ two morphism of $\Bbbk H$-modules, it is straightforward to verify that the following diagram commutes 
\[
\xymatrix{
\mathcal{F}(\mathbb{V})\otimes_{A_{par}(\Bbbk G)}\mathcal{F}(\mathbb{W}) \ar[rr]^-{\varphi_{\mathbb{V},\mathbb{W}}} \ar[dd]_{\mathcal{F}(f)\otimes_{A_{par}(\Bbbk G)}\mathcal{F}(g)} & & \mathcal{F}(\mathbb{V}\otimes \mathbb{W}) \ar[dd]^{\mathcal{F}(f\otimes g)} \\
& & \\
\mathcal{F}(\mathbb{V}')\otimes_{A_{par}(\Bbbk G)}\mathcal{F}(\mathbb{W}') \ar[rr]_-{\varphi_{\mathbb{V}',\mathbb{W}'}} & & \mathcal{F}(\mathbb{V}'\otimes \mathbb{W}')
}
\]

Now take the following morphism of $\kpar G$-modules
\[\begin{array}{cccl}
    \varphi: & A_{par}(\Bbbk G) & \longrightarrow & M_H\otimes \Bbbk = \mathcal{F}(\Bbbk) \\
     & P_X & \longmapsto &  P_H\otimes 1\bool{X=H},
\end{array}
\]
where $\Bbbk$ is the trivial $\Bbbk H$-module. It is easy to see that $\varphi$ is a morphism of $\Bbbk_{par}G$-modules, but clearly, it is not an isomorphism. To conclude, the natural transformation $\varphi_{\mathbb{V},\mathbb{W}}: \mathcal{F}(\mathbb{V})\otimes_{A_{par}(\Bbbk G)}\mathcal{F}(\mathbb{W}) \to \mathcal{F}(\mathbb{V}\otimes \mathbb{W})$ and the morphism $ \varphi:A_{par}(\Bbbk G) \to \mathcal{F}(\Bbbk) $ are the coherence maps of the functor $\mathcal{F}: { }_{\Bbbk H}\text{Mod} \to { }_{\Bbbk_{par} G}\text{Mod}$, what makes it monoidal, but not strongly monoidal, because the monoidal unit of ${}_{\Bbbk H}\text{Mod}$ is not isomorphic to the monoidal unit of ${}_{\Bbbk_{par}G}\text{Mod}$, which is $A_{par}(\Bbbk G)=\bigoplus_{X\in T} M_X$.




Now, let us prove that $\mathcal{F}$ is essentially injective. Recall that \[M_H = span_{\Bbbk } \{P_H\} \cong \Bbbk.\]
Then, for any $\Bbbk H$-module $\mathbb{V}$, $M_H \otimes \mathbb{V} \cong \mathbb{V}$ as a vector space. Consider $\mathbb{V}$ and $\mathbb{W}$ $\Bbbk H$-modules such that 
\[ 
\mathcal{F} (\mathbb{V}) \cong \mathcal{F}(\mathbb{W})
\] 
as $\Bbbk_{par}G$-modules. Then, from the structure of $\Bbbk_{par}G$-module defined over $M_H\otimes \mathbb{V}$ and $M_H \otimes \mathbb{W}$ we conclude that $\mathbb{V}$ is isomorphic to $\mathbb{W}$ as $\Bbbk H$-modules. 

Finally, $\mathcal{F}$ is a faithful functor. By the semisimplicity of the category ${}_{\Bbbk H}\text{Mod}$, given $\mathbb{V} = \displaystyle\bigoplus_{i=1}^{N} n_i \mathbb{V}_i$, with $\{ \mathbb{V}_i \}_{i=1}^N$ the family of all simple $\Bbbk H$-modules, and a morphism $f: \mathbb{V} \longrightarrow \mathbb{W}$, this morphism can be decomposed as a direct sum 
\[ f= \bigoplus_{i=1}^N n_i f|_{\mathbb{V}_i} : \mathbb{V} \cong \bigoplus_{i=1}^{N} n_i \mathbb{V}_i \longrightarrow \mathbb{W}, \]
where $f|_{\mathbb{V}_i}$ is either injective or zero, since $\mathbb{V}_i$ is simple. 

Consider $f,g : \mathbb{V} \longrightarrow \mathbb{W}$ two morphisms of $\Bbbk H $-modules such that $\mathcal{F}(f)=\mathcal{F}(g)$, then 
\[M_H \otimes f = M_H \otimes g : M_H \otimes \mathbb{V} \longrightarrow M_H \otimes \mathbb{W}.\]

Taking the restrictions of $f$ and $g$ to each simple summand of $\mathbb{V}$, one can see that 
\[ M_H \otimes f|_{\mathbb{V}_i}=M_H\otimes g|_{\mathbb{V}_i}. 
\]
This implies that $f|_{\mathbb{V}_i} = g|_{\mathbb{V}_i}$ for each simple summand, leading to the conclusion that  $f=g$ as morphisms of $\Bbbk H$-modules. Therefore, $\mathcal{F}$ is a faithful functor.
\qedhere
\end{proof}

For the next result, recall that any morphism of algebras $\psi :A \longrightarrow B$ induces a functor 
$$(\ \underline{\quad} \ )_{\psi}: { }_{B}\text{Mod} \longrightarrow { }_{A}\text{Mod}.$$ Indeed, given a $B$-module $M$, then $M$ can be made into an $A$-module by
\[
a\blacktriangleright m=\psi (a)\triangleright m.
\]
Similarly, given a morphism $f:{}_BM\rightarrow {}_BN$, then the same map $f$ is a morphism of $A$-modules between $M$ and $N$.

\begin{remark}
    In the next result, we are going to compute on the same space the structures of $\Bbbk_{par}H$-module and of $\Bbbk_{par}G$-module, with $H \leq G$. Then, we  denote elements in $\Bbbk_{par}G$ by unadorned brackets $[ \ ]$ and use  $[\ ]_H$ to denote elements in $\Bbbk_{par}H$. For the sake of clarity, we shall divide the proof of the Theorem in several stages in order to guide the reader through all the steps of this long proof. 
\end{remark}

\begin{theorem}[Matryoshka's Theorem] \label{matryoshka} Let $\Bbbk$ be an algebraically closed field of characteristic zero. Let $G$ be a finite abelian group and $H \leq G$. Then, there exists a functor
\[ \mathfrak{M}: {\ }_{\Bbbk _{par}H }\text{Mod} \longrightarrow  {\ }_{\Bbbk_{par}G}\text{Mod}.\]
Moreover, the functor $\mathfrak{M}$ is faithful, essentially injective, additive and monoidal.
\end{theorem}

\begin{proof}
{\textbf{ (1) The definition of the functor $\mathfrak{M}$.}}\\

First, in order to make the exposition more clear, let us consider $G$ to be a cyclic group of order $p^n$, where $p$ is a prime. The constructions made here will be extended for a general finite abelian group in Item (3). Without loss of generality, consider $n\geq 2$ (note that, for $n=1$, the group $G$ has prime order, then there will be no nontrivial subgroups to be analyzed),
\[ G = \mathbb{Z}_{p^n}= \langle a\ | \ a^{p^n}=e \rangle .\]
Let $H\leq G$, then $|H|=p^m$, for some $m\leq n$. For $m<n$, one can write
\[ H= \langle a^{p^{n-m}}\rangle \cong \mathbb{Z}_{p^m}. \]
Define 
\[ 
\begin{array}{rccl} \varphi : & G & \rightarrow & H \\
\, & a^t & \mapsto & a^{tp^{n-m}} .\end{array}
\]
From the definition, we have
\[ \varphi(a^s a^t) = \varphi(a^{s+t}) = a^{(s+t)p^{n-m}}=a^{sp^{n-m} + tp^{n-m}}=a^{sp^{n-m}}a^{tp^{n-m}}= \varphi(a^{s})\varphi(a^{t}) .\]
Therefore $\varphi$ is a homomorphism of groups.

The composition of a partial representation and a homomorphism of groups is also a partial representation, so the map
\[ \begin{array}{rccc} \pi:= [\ ]_{{ }_H}\circ \varphi: & G & \rightarrow & \Bbbk_{par}H \\
\, & a^t & \mapsto & \left[ a^{tp^{n-m}}\right]_H \end{array} \]
is a partial representation. By the universal property of $\Bbbk_{par}G$, there exists a unique morphism of algebras \[\overline{\pi}:\Bbbk_{par}G \longrightarrow \Bbbk_{par}H\]such that $\overline{\pi}\circ [\ ] = \pi = [\ ]_{ {}_H } \circ \varphi$, that is, for every $a^t \in G$, 
\[ \overline{\pi }([ a^t ])=[ \varphi(a^t) ]_{ { }_H }= [ a^{tp^{n-m}} ]_{ { }_H } . \]

The morphism of algebras $\overline{\pi}$ induces the functor \[ \mathfrak{M}=: (\ \underline{\quad} \ )_{\overline{\pi}}: {\ }_{\Bbbk _{par}H }\text{Mod} \longrightarrow  {\ }_{\Bbbk_{par}G}\text{Mod}. \]

Fixing the notation, given a  $\Bbbk_{par}H$-module $M$, we denote its $\Bbbk_{par}H$-module structure by 
$\xi \triangleright m$, for any $\xi \in \Bbbk_{par}H$ and $m\in M$. The induced $\Bbbk_{par}G$-module structure will be denoted by 
$\zeta \blacktriangleright m = \overline{\pi}(\zeta) \triangleright m$ for all $ \zeta \in \Bbbk_{par}G$ and $m \in M$. Explicitly, for $[ a^t] \in \Bbbk_{par}G $ and $m \in M$, we write $[a^t]\blacktriangleright m = [a^{tp^{n-m}}]_{ { }_H } \triangleright m$. 

Note that, the functor $\mathfrak{M} =\left( \underline{\quad} \right)_{\overline{\pi}}$ is, by construction, essentially injective, faithful and additive.\\

{\textbf{(2) The functor $\mathfrak{M}$ associates simple objects in ${}_{\Bbbk_{par}H}\text{Mod}$ to simple objects in ${}_{\Bbbk_{par}G}\text{Mod}$}}\\

Consider $X \in \mathcal{P}_{e}(H)$ and denote by $K=H_X$ its isotropy subgroup in $H$. Write $K= \langle a^{p^{n-l}} \rangle \cong \mathbb{Z}_{p^l}$, for some $l\leq m\leq n$. 

For $a^{tp^{n-l}}\in K$, one can write $a^{tp^{n-l}}=(a^{tp^{n-m}})^{p^{m-l}} \in H$. Then, the isotropy subgroup $K$ can be expressed in terms of generators as 
\[ K = \big\langle  a^{p^{n-l}}  \big\rangle = \big\langle \big(  a^{p^{n-m}}\big)^{p^{m-l}}   \big\rangle \leq H.\]

   Furthermore, $K$ being the isotropy group of $X$, one can write the subset $X$ explicitly as
   \[ X=K\cup \big( \bigcup\limits_{i \in \mathfrak{J} } a^{-ip^{n-m}}K\big),\]
   for some subset $\mathfrak{J} \subseteq \{1, \dots , p^{m-l}-1  \}$. 
   
   Now, consider $\omega \in \Bbbk$ being a primitive $p^l$-th root of the unity, and the irreducible representation of the subgroup $K$
   \[
   \begin{array}{rccl} \pi_{\omega}: & K & \rightarrow &\Bbbk^\times \\
   \, & a^{tp^{n-l}} & \mapsto & \omega^t .\end{array}
   \]
   Denoting by $\Bbbk_{\omega}$ the one dimensional vector space carrying the representation $\pi_{\omega}$, take the simple $\Bbbk_{par}H$ module
   $$M=M_X \otimes \Bbbk_{\pi_{\omega}} $$
   with the usual action, that is, for $ \Big[  a^{t{p^{n-m}}} \Big]_{{ }_H} \in \Bbbk_{par}H$ and \\
   $P_{a^{sp^{n-m}}X} \otimes 1 \in M$, we have 
   \begin{equation}\label{eq:acao-kparH-mod}
   \Big[  a^{t{p^{n-m}}} \Big]_{{ }_H}   \triangleright (P_{a^{sp^{n-m}}X} \otimes 1) = 
        \left( P_{a^{rp^{n-m}}X}\otimes \omega^{q} \right) \;   [\! [a^{-rp^{n-m}} \in X]\! ], 
   \end{equation}
   where $t+s=q{p^{m-l}}+r, \ 0 \leq r < p^{m-l}  $.

Using the induction functor $\mathfrak{M} =\left( \underline{\;} \right)_{\overline{\pi}}$, we conclude that, for any $[a^t] \in \Bbbk_{par}G$ and every $P_{a^{sp^{n-m}}X} \otimes 1 \in M$, we have
\begin{equation}
    [ a^t]\blacktriangleright \big(  P_{a^{sp^{n-m}}X}\otimes 1 \big) = [ a^{tp^{n-m}} ]_{ { }_H } \triangleright \big( P_{a^{sp^{n-m}}X} \otimes 1 \big) .
\end{equation}
turning $M$ into a $\Bbbk_{par}G$-module whose structure is defined by \eqref{eq:acao-kparH-mod}).

Now, we can verify that the simple $\Bbbk_{par}H$-module $M=M_X \otimes \Bbbk_{\pi_{\omega }}$, is isomorphic, as $\Bbbk_{par}G$-module,  to the simple $\Bbbk_{par}G$-module 
$N=M_{\varphi^{-1}(X)} \otimes { \Bbbk  }_{\widehat{\pi}_{\omega}}$. Where $\widehat{\pi}_{\omega} :L \rightarrow \Bbbk^{\times}$ is an irreducible representation of the subgroup 
$$ L=\varphi^{-1}(K)=  \langle a^{p^{m-l}} \rangle = \langle a^{p^{n-(n-m+l)}} \rangle   \cong \Z_{p^{n-m+l}}  .$$
Explicitly, the representation $\widehat{\pi}_{\omega}$ is written as
\[
\begin{array}{rccc} \varphi^\ast(\pi_{\omega}): & L & \rightarrow & \Bbbk^{\times}\\
\, & a^{tp^{m-l}} & \mapsto & \omega^t .\end{array}
\]

Recall that \[ X = K\cup \big( \bigcup\limits_{i \in \mathfrak{J} } a^{-ip^{n-m}}K\big), \]
for a subset $\mathfrak{J} \subseteq \{1, \dots , p^{m-l}-1  \}$. The pre-image, $Y=\varphi^{-1}(X)$ is then the subset
\[ Y = L\cup \big( \bigcup\limits_{i \in \mathfrak{J} } a^{-i}L\big),\]
whose isotropy subgroup is $G_Y= L$.

We need to verify that the $\Bbbk_{par} G$ modules
\[
\mathfrak{M} =M_X \otimes \Bbbk_{\pi_\omega}
\]
and
$$  N = M_Y \otimes \Bbbk _{ \varphi^\ast (\pi_\omega)  }  $$
are isomorphic.

Recall that $N$ is a simple $\Bbbk_{par}G$-module, as we have seen before, with explicit structure given by
\[ [a^{t}] \bullet \left(P_{ a^{s}Y } \otimes 1 \right) =  
    \left(P_{ a^{r}Y } \otimes \omega ^{q} \right)\; [\! [a^{-r} \in Y]\! ]   \]
where $t+s = qp^{m-l}+r$, $0 \leq r< p^{m-l}$.

Consider the map  
\[ \begin{array}{rccc} f: & M_X \otimes \Bbbk_{\pi_\omega} & \rightarrow & M_Y \otimes \Bbbk_{\varphi^\ast(\pi_\omega)} \\
  \, & P_{a^{ tp^{n-m} }X  }\otimes 1 & \mapsto & P_{ a^{t}Y } \otimes 1 . \end{array}\]
This map is a isomorphism of $\Bbbk_{par}G$-modules between $M$ and $N$. Indeed, for all $a^t \in G$, and $P_{a^{sp^{n-m}}X }\otimes 1 \in M_X\otimes \Bbbk _{\pi_\omega}$, we have

\begin{align*}
    f\big(  [a^t] \blacktriangleright P_{a^{sp^{n-m}}X }\otimes 1 \big) &= f\big( [a^{tp^{n-m}}]_{ { }_H } \triangleright P_{a^{sp^{n-m}}X }\otimes 1 \big) \\
    &= f\big( P_{a^{rp^{n-m}}X }\otimes \omega^{q}  \big) \; [\! [a^{ -rp^{n-m} } \in X]\! ]\\
    &= \left( P_{a^{r}Y }\otimes \omega^{q}\right) [\! [ a^{ -rp^{n-m} } \in X]\!] 
\end{align*}
where $t+s = qp^{m-l} +r$. 
On the other hand
\begin{align*}
    [a^t] \bullet f( P_{a^{sp^{n-m}}X }\otimes 1 ) &= [a^t] \bullet  \left(P_{a^{s}Y }\otimes 1\right) \\
    &= \left( P_{a^{r}Y }\otimes \omega^{q} \right) \; [\! [ a^{ -r } \in Y]\! ]
\end{align*}
because $t+s=qp^{m-l} + r$. And we have that
\[ a^{-rp^{n-m}} \in X \Longleftrightarrow a^{-r} \in Y.\] 
Indeed, $a^{-rp^{n-m}}\in X=K\cup \big( \bigcup\limits_{i \in \mathfrak{J} } a^{-ip^{n-m}}K\big)$. \\
If $a^{-rp^{n-m}} \in K$, then $\varphi(a^{-r}) \in K$, that is, $a^{-r} \in \varphi^{-1}(K)=L \subseteq Y$. If $a^{-rp^{n-m}} \in  a^{-ip^{n-m}}K$, for some $i \in \mathfrak{J}$, we have  $ a^{(-r+i)p^{n-m}} =\varphi(a^{-r+i}) \in K$, therefore $a^{-r+i} \in \varphi^{-1}(K)=L$, that is, $a^{-r} \in a^{-i} L$,
for the same $i\in \mathfrak{J}$, therefore $a^{-r} \in Y$, leading to the conclusion that $f$ is a morphism of $\Bbbk_{par}G$-modules. 

Now, for proving that $f$ is an isomorphism of $\Bbbk_{par}G$-modules. The $\Bbbk_{par}G$-module $N$ is simple, then $f$ is surjective. In order to prove that $f$ is also injective, it is enough to verify that $M$ is a simple $\Bbbk_{par}G$-module. We have that $M$  is already a simple $\Bbbk_{par}H$-module. Then any element of $M$ generates the entire $\Bbbk_{par}H$-module $M$. In fact, fix an arbitrary $r\in \mathfrak{J}$ then for any element
\[
x=\sum_{i \in \mathfrak{J}}  P_{a^{ip^{n-m}}X} \otimes \alpha_i,\in M_X \otimes \Bbbk_{\pi_\omega} ,
\]
we have
\begin{eqnarray*}
    x & = & \sum_{i \in \mathfrak{J}}  P_{a^{ip^{n-m}}X} \otimes \alpha_i \\
     & = & \left( \sum_{i \in \mathfrak{J}} \alpha_i \left[ a^{(i-r)p^{n-m}}\right]_{{}_H} \right) \triangleright \left( P_{a^{rp^{n-m}}X} \otimes 1\right) \\
     & = & \left( \sum_{i \in \mathfrak{J}} \alpha_i [a^{(i-r)}] \right)
     \blacktriangleright \left( P_{a^{rp^{n-m}}X} \otimes 1\right).
\end{eqnarray*}
Therefore, $M$ is a simple $\Bbbk_{par}G$-module, concluding the proof that $f$ is an isomorphism.\\

{\textbf{(3) The construcion of the functor $\mathfrak{M}$ for the case of a general finite abelian group.}}\\

For the general case, consider a finite abelian group $G$. The Theorem of elementary divisors states that there are prime numbers $p_1, \ldots , p_k$, not necessarily distinct, and positive integer numbers $n_1, \ldots , n_k$ such that 
\[ G\cong \Z_{p_1^{n_1}}\times \dots \times \Z_{p_k^{n_k}}.\]
In terms of generators, there exist elements $a_1, \ldots a_k \in G$ such that
\[ G=\{ a_{1}^{t_1}\dots a_{k}^{t_k}  \ | \ 0\leq t_i \leq p_{i}^{n_i}-1 , \ i \in \{1, \dots , k \}  \}.\]
One can write the elements of any subgroup $H \leq G$ in terms of the same generators
 \begin{eqnarray*}
     H & = & \{ a_{1}^{s_1 p_1^{n_1-m_1}}\dots a_{k}^{s_kp_k^{n_k-m_k}}  \ | \ 0\leq s_i \leq p_{i}^{m_i}-1 , \ i \in \{1, \dots , k \}  \} \\
     & \cong & \Z_{p_1^{m_1}}\times \dots \times \Z_{p_k^{m_k}},
 \end{eqnarray*}
 with $m_i \leq n_i$, for all $i$. 

 Using the same map
 \[
 \begin{array}{rrcl} \varphi : & G & \rightarrow & H \\
 \, & a_{1}^{t_1}\dots a_{k}^{t_k} & \mapsto & a_{1}^{t_1 p_1^{n_1-m_1}}\dots a_{k}^{t_kp_k^{n_k-m_k}} \end{array} ,
 \]
 one can define the partial representation $\pi=[\; ]_H \circ \varphi :G \rightarrow \Bbbk_{par}H$, inducing the morphism of algebras $\overline{\pi} :\Bbbk_{par}G \rightarrow \Bbbk_{par}H$ and finally, inducing the functor $\mathfrak{M}=\left( \underline{\quad} \right)_{\overline{\pi}}: {}_{\Bbbk_{par}H}\text{Mod} \rightarrow {}_{\Bbbk_{par}G} \text{Mod}$. By construction, this functor is essentially injective, faithful and additive.

Consider $X \in \mathcal{P}_e (H)$ with isotropy group $K$. As $K$ is a subgroup of $H$, one can also describe its elements in terms of the generators $a_1, \ldots ,a_n$, then \begin{eqnarray*} 
K & = & \{ a_{1}^{x_1 p_1^{n_1-l_1}}\dots a_{k}^{x_kp_k^{n_k-l_k}}  \ | \ 0\leq x_i \leq p_{i}^{l_i}-1 , \ i \in \{1, \dots , k \}  \} \\
& \cong & \Z_{p_1^{l_1}}\times \dots \times \Z_{p_k^{l_k}},
\end{eqnarray*}
with $l_i \leq m_i \leq n_i$. 

Take elements $\omega_{i } \in \Bbbk^{\times}$ such that $\omega_{i}^{p_i^{l_i}}-1=0$, for $i \in \{ 0, 1, \dots , k \}$ and define the irreducible representation $\pi_{\omega_1 , \ldots, \omega_k}: K\longrightarrow \Bbbk^\times$ given by
\[
\pi_{\omega_1 , \ldots, \omega_k} \left( a_{1}^{x_1 p_1^{n_1-l_1}}\dots a_{k}^{x_kp_k^{n_k-l_k}}\right) =\omega_1^{x_1} \cdots \omega_k^{x_k}.
\]
Denoting by $\Bbbk_{\pi_{\omega_1 , \ldots ,\omega_k}} =\Bbbk$ the vector space upon which the irreducible representation $\pi_{\omega_1 , \ldots, \omega_k}$ acts, we have that $M_X \otimes \Bbbk_{\pi_{\omega_1 , \ldots, \omega_k}}$ is a simple $\Bbbk_{par}H$-module via
\begin{align*}
 &[ a_{1}^{t_1p_1^{n_1-m_1}}\dots a_{k}^{t_k p_k^{n_k-m_k} }]_{ { }_{H} } \triangleright \bigg( P_{a_1^{s_1p_1^{n_1-m_1}} \dots \ a_{k}^{s_kp_k^{n_k-m_k}} X } \otimes 1  \bigg) =\\
 &= P_{a_1^{r_1p_1^{n_1-m_1}} \dots \ a_{k}^{r_kp_k^{n_k-m_k}} X } \otimes \ \omega_{1}^{q_1}\dots \omega_{k}^{q_k}, 
\end{align*}
where $t_i+s_i = q_ip_i^{m_i-l_i}+r_i$.\\

Write \[ X= K \cup \Bigg(  \bigcup_{s_1,...,s_k\in \mathfrak{J}} a_{1}^{s_{1}p_1^{n_1-m_1}}\dots \ a_{k}^{s_{k}p_k^{n_k-m_k} }K \Bigg), \]
with $\mathfrak{J} \subseteq \{ 0\leq s_{i}\leq  p_i^{m_i-l_i}-1 \; | \; i\in \{1 , \ldots , k\} \}$. 

Consider the morphism $\varphi: G \longrightarrow H$ defined by \[ \varphi(a_1^{t_1}\dots a_k^{t_k})=a_1^{t_1 p_1^{n_1-m_1}} \cdots a_k^{t_k p_k^{n_k-m_k}}.\]
Take the preimages $L=\varphi^{-1}( K )\cong \Z_{p_1^{n_1-m_1+l_1}} \times \dots \times \Z_{p_k^{n_k-m_k+l_k}} $ and $Y=\varphi^{-1}(X)\subseteq G$, written as
\[ Y= L \cup \Bigg(  \bigcup_{s_1,...,s_k\in \mathfrak{J}} a_{1}^{s_{1}}\dots \ a_{k}^{s_{k} }L \Bigg), \]
Again, the image $\mathfrak{M}(M_X \otimes \Bbbk_{\pi_{\omega_1, \dots, \ \omega _k}}) $ of the simple $\Bbbk_{par}H$-module $M_X \otimes \Bbbk_{\pi_{\omega_1, \dots, \ \omega _k}}$ is isomorphic to a simple $\Bbbk_{par}G$-module. Indeed, we have the following morphism
\[ \begin{array}{rccl} f: & M_X \otimes \Bbbk_{\pi_{\omega_1, \dots, \ \omega _k}} & \longrightarrow & M_Y \otimes \Bbbk_{\varphi^\ast(\pi_{\omega_1, \dots , \ \omega_k })} \\
\, & P_{a_1^{s_1p_1^{n_1-m_1}} \dots \ a_k^{s_kp_k^{n_k-m_k}} X}\otimes 1 & \longmapsto & P_{a_1^{s_1}\dots \ a_k^{s_k}Y} \otimes 1 .\end{array}
\]
By a similar reasoning, one can prove that $f$ is a isomorphism of $\Bbbk_{par}G$-modules. Therefore, the functor $\mathfrak{M}$ maps simple objects of the category of $\Bbbk_{par}H$-modules into simple modules of the category of $\Bbbk_{par}G$-modules. \\


{\textbf{(4) The functor $\mathfrak{M}$ is monoidal.}}\\

To see that $\mathfrak{M}$ is monoidal, first recall that if $X \in \mathcal{P}_e(G)$, then 
\[
\overline{\pi}(P_X) = \prod_{x\in X}\varepsilon_{\varphi(x)} \prod_{y\notin X}(1_{\kpar H}-\varepsilon_{\varphi(y)}),
\]
which will vanishes if there exist $x \in X$ and $y \notin X$ such that $\varphi(x) = \varphi(y)$, or $yx^{-1} \in \ker\varphi$, that is $y = (yx^{-1})x \in \ker\varphi X$. In other words, $\overline{\pi}(P_X) \neq 0$ if, and only if, $\ker\varphi \subseteq G_X$, and
\[
\overline{\pi}(P_X) = \bool{\ker\varphi \subseteq G_X}P_{\varphi(X)}.
\]
 Moreover, if we have $X,Y \in \mathcal{P}_e(G)$ such that $\ker\varphi \subseteq G_X\cap G_Y$ and $\varphi(X) = \varphi(Y)$ then $X=Y$.

 Collecting the above, given $M$ and $N$ two $\kpar H$-module, we have
 \begin{eqnarray*}
\mathfrak{M}(M)\otimes_{A_{par}(\Bbbk G)}\mathfrak{M}(N) & \simeq &  \sum_{X \in \mathcal{P}_e(G)}(P_{X}\blacktriangleright M)\otimes (P_{X}\blacktriangleright N) \\
& = & \sum_{{{X \in \mathcal{P}_e(G)}\atop{\ker\varphi \subseteq G_X}}}(P_{\varphi(X)}\triangleright M)\otimes (P_{\varphi(X)}\triangleright N) \\
& = & \sum_{Z \in \mathcal{P}_e(H)}(P_{Z}\triangleright M)\otimes (P_{Z}\triangleright N) \\
& \simeq & M\otimes_{A_{par}(\Bbbk H)} N,
\end{eqnarray*}
So these composition give us an linear bijection, given by 
\[
\begin{array}{cccc}
    \mu_{M,N}: & \mathfrak{M}(M)\otimes_{A_{par}(\Bbbk G)}\mathfrak{M}(N) & \longrightarrow & \mathfrak{M}(M\otimes_{A_{par}(\Bbbk H)}N) \\
     & m\otimes_{A_{par}(\Bbbk G)} n & \longmapsto & m\otimes_{A_{par}(\Bbbk H)} n
\end{array}
\]
to see that $\mu_{M,N}$ is $\kpar G$-linear first recall that, the monoidal structure of the category of modules of an Hopf Algebroid is given by the comultiplication which, in our case, for any $g \in G$, is written as 
$$\Delta_r([g]) = \Delta_l([g]) = [g]\otimes_{A_{par}(\Bbbk G)}[g].$$
Finally, for any $g \in G$, $m \in \mathfrak{M}(M)$ and $n \in \mathfrak{M}(N)$, 
\[
\begin{array}{ccl}
    \mu_{M,N}\bigl([g]\triangleright (m\otimes_{A_{par}(\Bbbk G)}n) \bigr) & = & \mu_{M,N}\bigl(([g]\blacktriangleright m)\otimes_{A_{par}(\Bbbk G)}([g]\blacktriangleright n) \bigr)  \\
     & = & \mu_{M,N}\bigl(([\varphi(g)]_H\triangleright m)\otimes_{A_{par}(\Bbbk G)}([\varphi(g)]_H\triangleright n) \bigr)  \\
     & = & ([\varphi(g)]_H\triangleright m)\otimes_{A_{par}(\Bbbk H)}([\varphi(g)]_H\triangleright n) \\
     & = & [\varphi(g)]_H\triangleright (m\otimes_{A_{par}(\Bbbk H)}n) \\
     & = & [g]\blacktriangleright (m\otimes_{A_{par}(\Bbbk H)}n) \\
     & = & [g]\blacktriangleright \mu_{M,N}\bigl(m\otimes_{A_{par}(\Bbbk G)}n\bigr).
\end{array}
\]
So we have a $\kpar G$-isomorphism $\mu_{M,N}:\mathfrak{M}(M)\otimes_{A_{par}(\Bbbk G)}\mathfrak{M}(N) \to \mathfrak{M}(M\otimes_{A_{par}(\Bbbk H)}N)$, now to see that this isomorphism is natural on $M$ and $N$, consider $f:M \to M'$ and $g:N \to N'$ morphism of $\kpar H$-modules. It is straightforward to verify that the following diagram commutes 
\[
\xymatrix{
\mathfrak{M}(M)\otimes_{A_{par}(\Bbbk G)}\mathfrak{M}(M) \ar[rr]^-{\mu_{M,N}} \ar[dd]_{\mathfrak{M}(f)\otimes_{A_{par}(\Bbbk G)}\mathfrak{M}(g)} & & \mathfrak{M}(M\otimes_{A_{par}(\Bbbk H)} N) \ar[dd]^{\mathfrak{M}(f\otimes_{A_{par}(\Bbbk H)} g)} \\
& & \\
\mathfrak{M}(M')\otimes_{A_{par}(\Bbbk G)}\mathfrak{M}(M') \ar[rr]_-{\mu_{M',N'}} & & \mathfrak{M}(M'\otimes_{A_{par}(\Bbbk H)} N').
} 
\]
Therefore, we have that $\mu_{-,-}:\mathfrak{M}(-)\otimes_{A_{par}(\Bbbk G)}\mathfrak{M}(-) \to \mathfrak{M}(-\otimes_{A_{par}(\Bbbk H)} -)$ is a natural isomorphism. 

However, this functor is not strongly monoidal. Recall that the unit object $A_{par}(\Bbbk H)$ is an $\kpar H$-module, with the partial representation given by 
\[
[h]_H\triangleright (\varepsilon_{k_1}\ldots\varepsilon_{k_n}) = \varepsilon_{hk_1}\ldots\varepsilon_{hk_n}\varepsilon_h
\]
for all $h, k_1,...,k_n \in H$. Consider the following map
\[
\begin{array}{rccc}
    \eta = \overline{\pi}|_{A_{par}(\Bbbk G)}: & A_{par}(\Bbbk G) & \longrightarrow & \mathfrak{M}(A_{par}(\Bbbk H)) \\
     & \varepsilon_{g_1}\ldots\varepsilon_{g_n} & \longmapsto & \varepsilon_{\varphi(g_1)}\ldots\varepsilon_{\varphi(g_n)}.
\end{array}
\]
It is easy to see that $\eta$ is a morphism of $\kpar G$-modules. In fact, consider $g \in G$ and notice that
\[
\begin{array}{ccl}
    \eta\bigl([g]\triangleright \varepsilon_{g_1}\ldots\varepsilon_{g_n}\bigr) & = & \eta\bigl(\varepsilon_{gg_1}\ldots\varepsilon_{gg_n}\varepsilon_g\bigr) \\
     & = & \varepsilon_{\varphi(gg_1)}\ldots\varepsilon_{\varphi(gg_n)}\varepsilon_{\varphi(g)} \\
     & = & [\varphi(g)]_H\triangleright\bigl(\varepsilon_{\varphi(g)}\ldots\varepsilon_{\varphi(g)}\bigr) \\
     & = & [g]\blacktriangleright\eta\bigl(\varepsilon_{g_1}\ldots\varepsilon_{g_n}\bigr),
\end{array}
\]
for all $g,g_1,...,g_n \in G$. The morphism $\eta$ is an isomorphism only when $H=G$.

To conclude, the natural isomorphism $\mu_{M,N}:\mathfrak{M}(M)\otimes_{A_{par}(\Bbbk G)}\mathfrak{M}(N) \to \mathfrak{M}(M\otimes_{A_{par}(\Bbbk H)}N)$ and the morphism $\eta:A_{par}(\Bbbk G) \to \mathfrak{M}(A_{par}(\Bbbk H))$ are the coherence maps of the functor $\mathfrak{M}: {\ }_{\Bbbk _{par}H}\text{Mod} \longrightarrow  {\ }_{\Bbbk_{par}G}\text{Mod}$, what makes it monoidal.

\end{proof}

\section{Conclusions and outlook}

From the results presented in this article, one can conclude that the monoidal structure of the category of ${\Bbbk_{par} G}$-modules reflects the structure of the categories of ${\Bbbk H}$-modules for all subgroups $H\leq G$ (the Christmas Tree's Theorem). This is, in certain sense, expected by the very structure of $\Bbbk_{par} G$ as a product of matrix algebras $M_n (\Bbbk H)$, for subgroups $H\leq G$. Moreover, the Matryoshka's Theorem highlights, for finite abelian groups, the interconnection between the categories of $\Bbbk_{par} G$-modules and the categories of $\Bbbk_{par}H$-modules for all subgroups $H\leq G$. We expect that the Matryoshka's Theorem can be extended for finitely presented abelian groups. For the case of non abelian finite groups, the general result doesn't appear to be straightforward, but we expect some weakened version of this result.

In \cite{BHVactegories}, it was proved that, for any Hopf algebra $H$, the category of partial $H$-modules is a module category (actegory) over the monoidal category ${}_H \text{Mod}$. Then the category of $\Bbbk_{par}G$-modules is a module category over the monoidal category ${}_{\Bbbk G}\text{Mod}$. On the other hand, for a finite group $G$ and a field $\Bbbk$ of characteristic zero, the category ${}_{\Bbbk G}\text{Mod}$ is a fusion category and its simple module categories are classified by subgroups $L\leq G$ and $2$-cocycles $\psi :L\times L \rightarrow \Bbbk^{\times}$ \cite{EGNObook}. Therefore, it is interesting to characterize the decomposition of the category of $\Bbbk_{par}G$-modules in terms of its simple components coming from the general classification of ${}_{\Bbbk G}\text{Mod}$-module categories.

\section{Acknowledgements}
Arthur Rezende Alves Neto is supported by FAPESC (Proc. no. 41001010001P6). Javier Méndez was supported during his PhD  course by CAPES (Proc. no.  88882.438824/2019-01, 88887.667127/2022-00 and 88887.854618/2023-00). We would like to thank to William Hautekiet for the first discussions about the proof of the Matryoshka's theorem, during his visit to the Department of Mathematics of UFSC in 2024. The authors E.B. and J. M. would like to express our immense gratitude to Gabriel Samuel de Andrade for the discussions and previous results that made this article possible.

\end{document}